\documentclass[12pt]{article}

\usepackage[english]{babel}

\usepackage[letterpaper,top=2cm,bottom=2cm,left=3cm,right=3cm,marginparwidth=1.75cm]{geometry}

\usepackage{amsmath}
\usepackage{graphicx}
\usepackage[colorlinks=true, allcolors=blue]{hyperref}
\usepackage{subfigure}
\usepackage{amssymb,amsfonts,amsmath,latexsym,amsthm}
\usepackage[nottoc]{tocbibind}

\title{Bibliography management: \texttt{natbib} package}
\author{Overleaf}
\date {}

\newcommand{\E}{\mathbb{E}}

\newtheorem{theorem}{Theorem}
\newtheorem{corollary}{Corollary}
\newtheorem{proposition}{Proposition}
\newtheorem{lemma}{Lemma}

\newtheorem{definition}{Definition}

\newtheorem{remark}{Remark}
\newtheorem{example}{Example}

\providecommand{\keywords}[1]
{
  \small	
  \textbf{\textit{Keywords---}} #1
}

\title{Adaptive estimation in regression models for weakly dependent data and explanatory variable with known density}

\author{Karine~Bertin$^{1}$, Lisandro Fermin$^{1}$ and  Miguel Padrino$^{2}$  \\
        \small $^{1}$CIMFAV-INGEMAT, Universidad de Valpa{r}a\'{\i}so, e-mail: \\
        \small karine.bertin@uv.cl; lisandro.fermin@uv.cl\\
        \small $^{2}$Departamento de Matem\'aticas, Instituto Venezolano de Investigaciones Cient\'ificas,\\ 
        \small e-mail: padrinomiguel@gmail.com\\
}

\begin{document}
\maketitle

\begin{abstract}
This article is dedicated to the estimation of the regression function when the explanatory variable is a weakly dependent process whose correlation coefficient exhibits exponential decay and has a known bounded density function. The accuracy of the estimation is measured using pointwise risk. A data-driven procedure is proposed using kernel estimation with bandwidth selected via the Goldenshluger-Lepski approach. We demonstrate that the resulting estimator satisfies an oracle-type inequality and it is also shown to be adaptive over H\"older classes. Additionally, unsupervised statistical learning techniques are described and applied to calibrate the method, and some simulations are provided to illustrate the performance of the method.
\end{abstract}


\tableofcontents

\section{Introduction}
In this article, we address the problem of nonparametric estimation using kernel estimators for regression functions in a univariate setting, based on observations $\{(X_{i},Y_{i})\}_{i=1:n}$ from the regression model $Y_i=r(X_i)+\varepsilon_i$, where the explanatory variables $\{X_{i}\}_{i=1:n}$ are weakly dependent and have a known common density $g$. We are interested in studying the asymptotic properties of the proposed estimator using Mean Squared Error (MSE) at a fixed point \(x\in\mathbb{R}\). For this risk, and in the context of dependence, we aim to obtain an adaptive estimator of the regression function \(r(\cdot)\) when it belongs to H\"older classes of regularity \(\beta>0\) (see \cite{tsybakov2004}). Adaptive estimation methods allow obtaining estimators that converge at an optimal rate over a class of general functions, without needing to know the regularity \(\beta\) of the function.

Among the various adaptive estimation methods, the Goldenshluger – Lepski method (G-L) stands out, introduced in \cite{goldenshluger2011bandwidth} and \cite{goldenshluger2014adaptive}. In the last decade, the G-L method has been employed in various statistical models, such as density estimation, regression, conditional density, white noise models, among others, to derive oracle-type inequalities and consequently, adaptive estimators. Especially in independent contexts (see, for example, \cite{BLR}, \cite{chichignoud2017adaptive}, \cite{goldenshluger2011bandwidth}, \cite{GL5}, \cite{Kouakou2020Lp}, \cite{Kouakou2023Sup}, \cite{lepski2014adaptive}). Additionally, it has also been applied in some dependent cases, as mentioned in \cite{asin2017adaptive}, \cite{bertin2017pointwise}, \cite{bertin2020adaptivebis}, \cite{bertin2020adaptive}, and \cite{comte2017adaptive}.

Regarding adaptive results in regression models and dependency contexts, we also mention \cite{baraud2001model} in the $L_2$ risk, who considers the case of $\beta$-mixing data and autoregressive models. Least squares estimators in families of models are considered, and it is shown that a penalized least squares estimator is adaptive. On the other hand, \cite{asin2017adaptive} considers adaptive non-parametric estimation of density and regression in contexts of independence and weak dependence, where the function to be estimated belongs to the Hilbert space $L_{2} := L_{2}[0, 1]$. The estimator is based on an orthogonal series approach, where the selection of the dimension parameter is entirely data-driven, inspired by the G-L method.

We propose in this article to use the procedure presented \cite{lepski2014adaptive}, where adaptive regression estimation is performed for independently and identically distributed data, in pointwise risk, and to extend the G-L method to the posed regression case in a weak dependence context. 
More precisely, we use the Nadaraya-Watson type estimator \begin{equation}\label{toto}
\hat{r}_h(x)=\frac{1}{n}\sum_{i=1}^nY_{i}K_h(x-X_{i})g^{-1}(X_{i})
\end{equation}
proposed in \cite{lepski2014adaptive}. 


First we prove that the estimator \eqref{toto} converges at the optimal rate  $n^{-\frac{\beta}{2\beta+1}}$ over H\"older classes of regularity $\beta$. To obtain this, we perform a classical study of the bias term and the exponential decay of the covariances of the explanatory variables allows for an efficient control of the estimator's variance. \\ 
Second we consider the family of regression estimators $\{r_h, h\in \mathcal{H}\}$ for  a well-chosen family of bandwidths $\mathcal{H}$. We select a bandwidth $\hat{h}$ using G-L method. We prove that the selected estimator $\hat{r}_{\hat{h}}$ satisfies an oracle-type inequality. The proof does an intensively use of the Bernstein inequality for dependent data proposed in \cite{doukhan2007probability}. \\ 
Finally, the oracle inequality allows us to demonstrate that the selected estimator, that does not depend on $\beta$, converges at the rate $\left(\frac{n}{\log n}\right)^{-\frac{\beta}{2\beta +1}}$ over the Hölder classes of regularity $\beta$. This rate is nearly optimal or minimax, except for the logarithmic multiplicative term. This additional logarithmic term is classical in the adaptive case and also appears, for example, in \cite{lepski2014adaptive}. 

Our contribution is as follows: we obtain the adaptive convergence rate for pointwise risks over a wide range of Hölder spaces in the context of weak dependence. This partially generalizes the results obtained in the i.i.d. case in \cite{lepski2014adaptive}. As far as we know, this is the first adaptive result based on the G-L method for pointwise estimation of the regression function in the context of dependent data. The selected estimator performs almost as well as the best estimator in a given finite family of kernel estimators.

Moreover, our data-driven procedure depends only on explicit quantities, which implies that this procedure can be directly implemented in practice. 
As a direct consequence, we obtain a new method for choosing a precise local bandwidth for kernel estimators. In particular, our method depends on the calibration of a parameter $\gamma$. Such calibration is performed using unsupervised statistical learning techniques. 
Additionally, a simulation study is conducted to illustrate the performance of our method.

The techniques developed in this research pave the way to address the problem of adaptive estimation, based on the G-L method, of the regression function in a univariate model when the explanatory variables are weakly dependent and the density function $g$ is \textbf{unknown}. This topic is addressed in a second work.

The remainder of this article is organized as follows. Section \ref{AAModelo} is dedicated to presenting our model and assumptions about the process \(\mathbb{X}\), and an example of a process that satisfies the established hypotheses is provided. The construction of our estimation procedure is developed in Section \ref{AAEstPro}. The main results of the article are presented in Section \ref{AAresultados}, which consists of two subsections: \ref{AAsubsec5.1:ResulHat(r)} where the bias, variance, consistency, and convergence rate of the estimator are established, and \ref{AASubSecDOraculo} where it is demonstrated that the kernel estimator in the selected bandwidth satisfies an oracle-type inequality, allowing us to establish that the kernel estimator in the selected bandwidth is adaptive. A simulation study is conducted in Section \ref{AAsimulacion} to explain the calibration technique of the method and illustrate its performance. Additionally, the article has three appendices: \ref{AARTecnicos} where all the constants obtained in the proofs are explicitly provided, and the technical results are stated and proved, \ref{AAControl:sesgoYvar} where the proofs of the control of the estimator's bias and variance are found, \ref{AAapp:desi} where three inequalities preceding the oracle-type inequality are proved.

\section{Model}\label{AAModelo}

We observe $\{(X_1, Y_1), \ldots, (X_n, Y_n)\}$, a sample satisfying
$$Y_i=r(X_i)+\varepsilon_i,\quad i=1,\ldots,n$$
where $n \geq 2$, the $X_i$ are identically distributed with known probability density function $g$, the $\varepsilon_i$ are independent and identically distributed with a normal distribution with zero mean and variance $\sigma^2 > 0$, the $X_i$ are independent of the $\varepsilon_i$, and $r$ is the regression function $r(x) = \mathbb{E}(Y | X = x)$.

Our goal is to estimate the function $r$ at the point $x \in \mathbb{R}$ using the observed sample $\{(X_1, Y_1), \ldots, (X_n, Y_n)\}$ under a weak dependence approach.\\

The quality estimation of an estimator $\hat{r}$ is measured using the mean squared error at a point $x \in \mathbb{R}$.
$$R(r,\hat{r},x)=\E \left(r(x)-\hat{r}(x)\right)^2.$$

On one hand, the following hypotheses are assumed, corresponding to bounds on the densities of the variables $X_i$ in the neighborhood of point $x$:
$$B(x)=[x-2/(\log n)^2, x+2/(\log n)^2].$$

\begin{description}
\item[$(H_1)$] The density $g$ of the $X_i$ is bounded on $B(x)$; i.e. 
      $$g_{\inf}\leq g(u) \leq g_{\sup}, \quad\forall u\in B(x)$$
    with $g_{\inf}$ and $g_{\sup}$ positive constants.
\item[$(H_2)$] The joint densities $g_{i,j}$ of $(X_i, X_j)$ are bounded  on $B(x)$; i.e.
      $$|g_{i,j}(u,v)|\leq Q, \quad\forall u,v\in B(x), $$
      where $Q$ is a positive constant.
\end{description}

On the other hand, we assume a weak dependency structure on the variables $X_i$. More precisely, for positive integers $u$ and $v$, we denote $i_{1:u}=(i_1,\ldots,i_u)\in\mathbb{Z}^u$ and $j_{1:v}=(j_1,\ldots,j_v)\in\mathbb{Z}^v$, then the random vectors $X_{i_{1:u}}=(X_{i_1},X_{i_2},\cdots, X_{i_u})$ and $X_{j_{1:v}} =(X_{j_1}, X_{j_2},\cdots, X_{j_v})$ are defined as values in $\mathbb{R}^u$ and $\mathbb{R}^v$, respectively. The function $q:\mathbb{Z}^u\times\mathbb{Z}^v \rightarrow \mathbb{Z}$ is defined by $q(i_{1:u},j_{1:v})=\min (j_{1:v}) - \max (i_{1:u})$, and $\Lambda_u$ is the class of functions $G_u:\mathbb{R}^u \rightarrow \mathbb{R}$ such that $\|G_u\|_{\infty} = \sup _{x\in\mathbb{R}^u} |G_u(x)| < \infty$. For a random process $\mathbb{X}=(X_i)_{i\in\mathbb{Z}}$, the correlation coefficient $\alpha(\mathbb{X})=\left(\alpha_k(\mathbb{X})\right)_{k\in\mathbb{N}}$ is defined by,

\begin{equation} \label{AACoefDevil}
    \alpha_k (\mathbb{X})=\sup_{u,v\in\mathbb{N}} \sup_{(i_{1:u},j_{1:v})\in\mathbb{Z}^u\times\mathbb{Z}^v q(i_{1:u},j_{1;v})\geq k} \sup_{G_u \in \Lambda_u} \sup_{G_v \in \Lambda_v} \frac{|cov (G_u (X_{i_{1:u}}),G_v (X_{j_{1:v}} ))|}{\Psi (u,v,G_u , G_v )}
\end{equation}

with $\Psi(u,v,G_u ,G_v )=4\|G_u\|_{\infty}\|G_v\|_{\infty}$. See Section 2 in \cite{bertin2017pointwise}.\\

We assume that the $\alpha(\mathbb{X})$ coefficient of the process $\mathbb{X}$ satisfies the following hypothesis.
\begin{description}
    \item[$(H_3)$] There exists $a\in ]0,1[$ such that
    $$\alpha_k (\mathbb{X})\leq a^k, \quad \forall k\in\mathbb{N}.$$
\end{description}

In what follows, we will denote by $\mathcal{L}$ the set of processes $\mathbb{X}$ that satisfy $(H_1)$, $(H_2)$, and $(H_3)$.\\


Next, two examples of processes that satisfy these hypotheses are provided.
\begin{example}\label{AAEjempAR1}
    We consider $\mathbb{Z}=\{Z_t\}_{t\geq 1}$, an autoregressive process of order 1, defined by $Z_t =\phi Z_{t-1}+\rho \xi_t$, with $\xi_t \sim N(0,1)$, $|\phi|<1$, and $\rho >0$. By recurrence, $Z_{t}=\phi^{n}Z_{t-n}+\rho\sum_{i=0}^{n-1}\phi^{i}\xi_{t-i}$, and taking $n\rightarrow \infty$ results in $Z_{t}=\rho\sum_{i=0}^{\infty}\phi^{i}\xi_{t-i}$, which is a centered and stationary Gaussian process with covariance function $\gamma_{Z}(k)=cov(Z_{t},Z_{t+k})=\frac{\rho^2}{1-\phi^2}\phi^k$ for $k\geq 0$. It can be shown that the process $\mathbb{Z}$ is $\alpha$-weakly dependent satisfying hypothesis $(H_{3})$, see \cite{doukhan1994mixing}.
\end{example}

\begin{example}\label{AAProcesoXenL}
    Denoting $\phi_{\mu,\sigma^2}(x)$ as the probability density function for $N(\mu,\sigma^{2})$ and $\Phi_{\mu,\sigma^2}(x)$ as the cumulative distribution function for $N(\mu,\sigma^{2})$. For $c\in\mathbb{R}^+$, the probability density and cumulative distribution functions of the truncated normal distribution in $[-c,c]$, with mean zero and variance one, are given by $g(x)=\frac{\phi_{0,1}(x)}{p}\mathbf{1}_{[-c,c]}(x)$, where $p=\Phi_{0,1}(c)-\Phi_{0,1}(-c)$, for each $x\in\mathbb{R}$, and
    $$G(x)=\left\{ \begin{array}{lcc}
         0 & if & x<-c  \\
         \\ \frac{1}{p}\left(\Phi_{0,1}(x)-\Phi_{0,1}(c) \right)& if & -c\leq x < c \\
         \\ 1 & if & x\geq c
    \end{array}
    \right.$$
    respectively. Moreover, the inverse function of $G$ is given by $G^{-1}(u)=\Phi_{0,1}^{-1}\left( pu+\Phi_{0,1}(c) \right)$ for each $u\in[0,1]$.

    It can be shown that the process $\mathbb{X}=\{X_t\}_{t\geq 1}$ defined by $X_{t}=\left(G^{-1}\circ\Phi_{0,\frac{\rho^2}{1-\phi^2}}\right)(Z_{t})$ satisfies $\mathbb{X}\in\mathcal{L}$, where the process $\mathbb{Z}=\{Z_{t}\}_{t\geq 1}$ with each 
    $Z_t$ following the distribution  $\Phi_{0,\frac{\rho^2}{1-\phi^2}}$ is the process from Example \ref{AAEjempAR1}.
\end{example}

Finally, we assume the following hypothesis about the regression function.
\begin{description}
        \item[$(H_4)$] There exists a constant $r_{\sup}>0$ such that
        $$|r(u)|\leq r_{\sup}, \quad \forall u\in B(x).$$
\end{description}

\section{Statistical Procedure}
\label{AAEstPro}
In this section, we present the kernel estimator of the regression function in the known density case, specify the family of bandwidths and describe the selection procedure based on the Goldenhluger-Lepski method.\\

We consider the regression function estimator given by
$$
\hat{r}_h(x)=\frac{1}{n}\sum_{k=1}^nY_kK_h(x-X_k)g^{-1}(X_k)
$$
where $h>0$ is the bandwidth,  $K_{h}(\cdot)=\frac{1}{h}K(\frac{\cdot}{h})$, and $K:\mathbb{R}\rightarrow \mathbb{R}$ is a kernel function satisfying
$\int K(u)du =1$, $\int uK(u)du =0$, and the following hypothesis.
\begin{description}
    \item[$(H_5)$] $K$ has support $[-1,1]$ and $\|K\|_{\infty}<+\infty$.
\end{description}
This hypothesis implies that $\|K\|_{1}<+\infty$ and $\|K\|_{2}<+\infty$.
    
To select the bandwidth $h$ of the estimator, we will use the Goldenshluger-Lepski (GL) method. We consider the family of bandwidths
$$\mathcal{H}=\{e^{-i}\}_{i=0}^{M}\cap [h_{\min},h_{\max}]$$
with $h_{\min} = \frac{(\log n)^8}{n}$, $h_{\max} = \frac{1}{(\log n)^2}$, and $M=\left[\log\left(\frac{n}{(\log n)^8}\right)\right]$. 
For $h,h'\in\mathcal{H}$, an oversmoothed auxiliary estimator is defined as
$$\hat{r}_{h,h'}(x)=\frac{1}{n}\sum_{k=1}^nY_kK_h\ast K_{h'}(X_k-x)g^{-1}(X_k).$$
We define, for $h\in\mathcal{H}$,
$$
A(h,x)= \max_{h'\in\mathcal{H}}\{|\hat{r}_{h,h'}(x)-\hat{r}_{h'}(x)|-V(h')\}_+
$$
where $V(h)$ is given by,

$$V(h)=\sqrt{ 2\gamma A_{1}}( \|K\|_1 + 1)( 1 + \delta_n ) \frac{(\log n)^\frac{1}{2}}{(nh)^\frac{1}{2}}$$
with  $\gamma>2$, $A_{1} = ((r_{\sup})^2+\sigma^2)(g_{\inf})^{-1}\|K\|_2^2$ and $\delta_n =(\log n)^{-\frac{1}{5}}$.\\

The GL procedure consists in selecting, based on the data, a bandwidth $\hat h$ from the family $\mathcal{H}$, given by
\begin{equation}\label{eq:h}
\hat{h}= \mbox{arg}\hspace{-0.05cm}\min_{h\in\mathcal{H}}\{A(h,x)+V(h)\},
\end{equation}
The resulting estimator, $\hat{r}_{\hat{h}}(x)$, satisfies an oracle inequality that allows demonstrating its adaptability.

\section{Results} \label{AAresultados}
This section is divided into two subsections. In the first one, we show that the kernel estimator of the regression function proposed in Section \ref{AAEstPro} satisfies the necessary conditions to be a good estimator. In the second subsection, it is shown that the kernel estimator at the bandwidth selected by the G-L method satisfies an Oracle-type inequality, which allows us to establish that this estimator is adaptive.

\subsection{Results for the estimator \texorpdfstring{$\hat{r}_h$}{rh} } \label{AAsubsec5.1:ResulHat(r)}

In this subsection, we study the consistency and convergence rate of the regression estimator. Results on the control of the bias and variance are performed in \ref{AAControl:sesgoYvar}. To control the bias and study the asymptotic unbiasedness, we introduce the definitions of a Hölder class and a kernel of order $m$.

\begin{definition}[Hölder class]
Let $\beta>0$ and $L>0$. The Hölder class $\Sigma(\beta,L)$ is defined as the set of all functions $f:\mathbb{R}\rightarrow\mathbb{R}$ such that the derivative $f^{\mathit{(l)}}$, with $\mathit{l}=\left \lfloor \beta \right \rfloor$, exists and 
\begin{equation*}
\left|f^{(\mathit{l})}(x)-f^{(\mathit{l})}(y)\right|\leqslant L\left|x-y\right|^{\beta-\mathit{l}}, \forall x,y\in \mathbb{R}
\end{equation*}
where $\left \lfloor \beta \right \rfloor=\max\left\{ n\in\mathbb{N},n<\beta\right\}$.
\end{definition}

\begin{definition}[Kernel of order $m$] Let $m\in\mathbb{N}$. It is said that $K:\mathbb{R}\rightarrow\mathbb{R}$ is a kernel of order $m$ if the functions $u\rightarrow u^{j}K(u)$, for $j=0,1,...,m$, satisfy
\begin{equation*}
\int_\mathbb{R}K(u)du=1, \int_\mathbb{R}u^{j}K(u)du=0, j=1,...,m.
\end{equation*}
\end{definition}
In the following two propositions, the bias and variance of the estimator $\hat{r}_h$ are bounded.
\begin{proposition}\label{AApro:sesgo} 
Let $\beta>0$ and $L>0$. We assume that $K$ is of order $m$, where $m\geq \mathit{l}=\left \lfloor \beta \right \rfloor$, and it satisfies $\int_\mathbb{R}|u|^{\beta}|K(u)|du<\infty$. Then, under hypothesis $(H_{1})$, if $r\in\Sigma(\beta,L)$, we have
$$\left|E\left[\hat{r}_{h}(x)\right]-r(x)\right|=\left|K_h \ast r (x) -r(x)\right|\leq h^{\beta}A_{0},$$
where $A_0 =\frac{L}{\mathit{l}!}\int_{\mathbb{R}}|u|^{\beta}|K(u)|du$. The estimator $\hat{r}_{h}$ is asymptotically unbiased as $h$ tends to $0$.
\end{proposition}

 \begin{proposition}\label{AApro:var}
  Under the assumptions $(H_1)$, $(H_2)$, $(H_3)$, $(H_4)$, and $(H_5)$, it holds that for $h\in (0, h_{max})$
  $$\mathbb{E}\left[(\hat{r}_h (x)-\mathbb{E}[\hat{r}_h (x)])^2\right] \leq \frac{A_{1}}{nh}+A_2\frac{(\log n)^{-\frac{1}{2}}}{nh},$$
  where $A_{1}=((r_{\sup})^2 +\sigma^2 )(g_{\inf})^{-1}\|K\|_2^2$, $A_2 = 2A_3 +\frac{4}{|\log a|}A_4$, $A_3 = (r_{\sup})^2 \|K\|_1^2 \left((g_{\inf})^{-2}Q+1\right)$ and $A_4 =\frac{4}{1-a}(r_{\sup})^2(g_{\inf})^{-2}\|K\|_{\infty}^2$.
\end{proposition}

Propositions \ref{AApro:sesgo} and \ref{AApro:var} imply that as $n\rightarrow\infty$, $h\rightarrow0$, and $nh\rightarrow \infty$, the mean squared error $MSE\left(\hat{r}_{h},r,x\right)=E\left[\left(\hat{r}_{h}(x)-r(x)\right)^{2}\right]\rightarrow 0$. In other words, $\hat{r}_{h}(x)$ is a consistent mean squared estimator (and therefore, consistent in probability) of $r(x)$.

 Now we will determine the convergence rate of the regression estimator $\hat{r}_{h}(x)$, as stated in the following theorem.

\begin{theorem}\label{AAveloconv}
Under assumptions of Propositions \ref{AApro:sesgo} and \ref{AApro:var}, it follows that if $r\in\Sigma(\beta,L)$, then the estimator $\hat{r}_{h_{*}}$, with 
$h_{*}=n^{-\frac{1}{2\beta+1}}$, satisfies:
$$E\left[\left(\hat{r}_{h_{*}}(x)-r(x)\right)^{2}\right]\leq C_* n^{-\frac{2\beta}{2\beta+1}},$$
where $C_*$ is a positive constant depending on the kernel $K$, of $\beta$, $L$, $r_{sup}$, $Q$, $g_{inf}$, $\sigma$, and $a$.
\end{theorem}
\begin{proof}
    Using Propositions \ref{AApro:sesgo} and \ref{AApro:var}, it is deduced that for $n\ge 4$,
    \begin{equation}\label{AAeq:mse}MSE\left(\hat{r}_{h},r,x\right)\leq h^{2\beta}A^{2}_{0}+\frac{A_{1}+A_2}{nh}.
    \end{equation}

    Substituting $h$ with $h_*$ in the above equation, we have:

    \begin{equation*}
    MSE\left(\hat{r}_{h_*},r,x\right)\leq (A_{0}^{2}+A_{1}+A_{2})n^{-\frac{2\beta}{2\beta+1}}
    \end{equation*}   
\end{proof}

Note that any bandwidth of the form $An^{-\frac{1}{2\beta+1}}$ allows obtaining the same convergence rate. In particular, the bandwidth that minimizes the term on the right in the inequality \eqref{AAeq:mse}, since it is of the form $h=\left( \frac{A_1 +A_2}{2\beta A_0^2} \right)^{\frac{1}{2\beta +1}}n^{-\frac{1}{2\beta+1}}$.\\



Based on the previous result, we have:
$$\sup_{r\in\Sigma(\beta,L)}E\left[\left(\hat{r}_{h_{*}}(x)-r(x)\right)^{2}\right]\leq C_* n^{-\frac{2\beta}{2\beta+1}},$$
and
$$\inf_{\hat{r}\in \Theta}\sup_{r\in\Sigma(\beta,L),\mathbb{X}\in\mathcal{L}}E\left[\left(\hat{r}(x)-r(x)\right)^{2}\right]\leq C_* n^{-\frac{2\beta}{2\beta+1}}, $$
where $\Theta$ is the set of all estimators of $r$, and we recall that $\mathcal{L}$ is the set of all processes $\mathbb{X}$ that satisfy $(H_1)$, $(H_2)$, and $(H_3)$.\\
In \cite{tsybakov2004} (see Section 2.5), it is shown that
$$\inf_{\hat{r}\in \Theta}\sup_{r\in\Sigma(\beta,L),\mathbb{X}\in\mathcal{L}}E\left[\left(\hat{r}(x)-r(x)\right)^{2}\right]\ge c n^{-\frac{2\beta}{2\beta+1}},$$
where $c$ is a constant depending on $\beta$, $L$, $g_{sup}$, and $\sigma^2$.

This allows us to conclude that the estimator $\hat{r}_{h_*}$ converges at the optimal rate over the H\"older class $\Sigma(\beta,L)$, which is given by
\begin{equation}\label{AAvc}
    v_{n}=n^{-\frac{\beta}{2\beta+1}}.
\end{equation}

 \subsection{Oracle Inequality and Adaptability of the Estimator}\label{AASubSecDOraculo} 
In this subsection we show that the kernel estimator at the selected bandwidth by the G-L method satisfies an Oracle-type inequality. Finally, by the usual techniques of the G-L method and applying the Oracle-type inequality we obtain that the estimator is adaptive. To achieve such objective the following two propositions are stated, their proofs can be seen in Appendix \ref{AAapp:desi}.\\
 
In the following proposition, we obtain a first inequality satisfied by the estimator $\hat{r}_{\hat{h}}(x)$ using mainly the definition of $\hat{h}$ given in \eqref{eq:h}.
 \begin{proposition}\label{AAPropDO}
Under hypothesis $(H_5)$, the estimator $\hat{r}_{\hat{h}}(x)$ satisfies for all $h\in\mathcal{H}$, 
 \begin{eqnarray}
 \nonumber
\lefteqn{  \left(\mathbb{E}\left(r(x)-\hat{r}_{\hat{h}}(x)\right)^2\right)^{\frac{1}{2}} }\\
\nonumber
&\leq& \left(\mathbb{E}\left(r(x)-\hat{r}_{h}(x)\right)^2\right)^\frac{1}{2}+2\left(\mathbb{E}(T_{1}^{2})\right)^{\frac{1}{2}}+2\left(\mathbb{E}(T_{2}^{2})\right)^{\frac{1}{2}}+ 2C(h)\|K\|_1 +2V(h)
\end{eqnarray}
where
$$T_1=\max_{h'\in\mathcal{H}}\{|\hat{r}_{h'}(x)-\mathbb{E}(\hat{r}_{h'}(x))|-V_1(h')\}_+, \quad 
T_2=\max_{h'\in\mathcal{H}}\{|\hat{r}_{h,h'}(x)-\mathbb{E}(\hat{r}_{h,h'}(x))|-V_2(h')\}_+,
$$
$$C(h)=\max_{u\in B(x)}\left|K_h\ast r(u)-r(u)\right|$$
and  $V_{1}(h')$, $V_{2}(h')$ such that $V(h')=V_{1}(h')+V_{2}(h')$, with

$$V_1 (h') = \sqrt{ 2\gamma A_{1}\frac{\log n}{nh'} } ( 1 + \delta_n ), \quad  V_2 (h') = \|K\|_1V_1 (h').$$ 
\end{proposition}

The proof is reported to Appendix \ref{AAdemopropDO}.

In the previous proposition, a bound for the square root of the pointwise mean squared error of the estimator $\hat{r}_{\hat{h}}(\cdot)$ is provided in terms of the bias and variance of the estimator $\hat{r}_{h}(\cdot)$, as well as $V_{1}(h')$, $V_{2}(h')$, $\E(T_{1}^{2})$, and $\E(T_{2}^{2})$. The proof of the proposition crucially relies on the fact that $V(h')=V_{1}(h')+V_{2}(h')$. The precise choices made for the expressions of $V_{1}(h')$ and $V_{2}(h)$ are fundamental for the proof of Proposition \ref{AAPropET_1^2yV_1} and allow establishing that $\mathbb{E}(T_1^2)$ and $\mathbb{E}(T_2^2)$ are negligible compared to the bias and variance terms of any estimator $\hat{r}_h$ for $h\in\mathcal{H}$.\\

To control the terms $\E(T_1^2 )$ and $\E(T_2^2 )$, a truncation operator is used, in addition to the weak dependence of the explanatory variable and the Bernstein inequality proposed in \cite{doukhan2007probability}.

\begin{proposition} \label{AAPropET_1^2yV_1}
 Under hypotheses $(H_1)$, $(H_2)$, $(H_3)$, $(H_4)$, and $(H_5)$, we have for sufficiently large $n$
 \begin{itemize}
 \item[(i)] $ \mathbb{E}(T_1^2) \leq A_5\frac{\log n}{n},$
\item[(ii)] $\mathbb{E}(T_2^2) \leq A_5\|K\|_1^2\frac{\log n}{n},$
\end{itemize}
where $A_5$ is an explicit constant (see Appendix \ref{AAConstantes}) that depends on the kernel $K$, $r_{sup}$, $Q$, $g_{inf}$, $\sigma$, $a$, and $\gamma$.

 
\end{proposition}

The proofs of item $(i)$ and $(ii)$ in Proposition \ref{AAPropET_1^2yV_1} are reported to Appendix \ref{AAdemoT1} and Appendix \ref{AAdemoT2} respectively.



\begin{theorem} \textbf{(Oracle Inequality)}\label{AAteo:D-Oraculo}
 Under hypotheses $(H_1)$, $(H_2)$, $(H_3)$, $(H_4)$, and $(H_5)$, the estimator $\hat{r}_{\hat{h}}(x)$ satisfies the following inequality,
\begin{eqnarray} \label{eq:oracle}
    \left( \mathbb{E}\left(r(x)-\hat{r}_{\hat{h}}(x)\right)^2\right)^{\frac{1}{2}} &\leq& \min_{h\in\mathcal{H}} \left( A_6 C(h) + A_7 ( 1 + \delta_n ) \frac{(\log n)^{\frac{1}{2}}}{(nh)^{\frac{1}{2}}} \right) + A_8\frac{(\log n)^\frac{1}{2}}{n^\frac{1}{2}}
\end{eqnarray}
where $A_6 =1+2\|K\|_1$, $A_7 =\sqrt{A_{1} +A_2} +2\sqrt{ 2\gamma A_{1}}( \|K\|_1 + 1)$, and $A_8$ is a positive constant depending on the kernel $K$, of $r_{sup}$, $Q$, $g_{inf}$, $\sigma$, $a$, and $\gamma$.
\end{theorem}

\begin{remark}
In Theorem \ref{AAteo:D-Oraculo}, we can see in \eqref{eq:oracle} that the estimator $\hat{r}_{\hat{h}}$ mimics the "oracle"; i.e., the best possible (but unknown) estimator $\hat{r}_{h}$ in the family $(\hat{r}_h)_{h\in\mathcal{H}}$, which minimizes the sum of bias $C(h)$ and standard deviation. When the family $\left(\hat{r}_h\right)_{h\in \mathcal{H}}$ is wide enough, estimators that satisfy oracle inequalities tend to be adaptive estimators. This is further demonstrated in Theorem \ref{T3:adapt}.
\end{remark}
\begin{proof}
Let us $h\in\mathcal{H}$. Using $|\mathbb{E}(\hat{r}_h (x)) - r(x)|=|K_h \ast r(x) -r(x)| \leq C(h)$ and Proposition \ref{AApro:var}, we have:
\begin{eqnarray} \label{AAECMTilderh}
    \nonumber
    \left(\mathbb{E}\left[ \left(r(x)-\hat{r}_{h}(x)\right)^2\right]\right)^\frac{1}{2} &\leq&  C(h)+
    \left( \frac{A_{1}}{nh}+A_2\frac{(\log n)^{-\frac{1}{2}}}{nh} \right)^\frac{1}{2} \\
    &\leq&  C(h)+ (A_{1}+A_{2})^{1/2} (nh)^{-1/2}
\end{eqnarray}

Now, by Proposition \ref{AAPropDO} and inequality \eqref{AAECMTilderh}, we have:

\begin{eqnarray*}
\nonumber
\lefteqn{\left(\mathbb{E}\left(r(x)-\hat{r}_{\hat{h}}(x)\right)^2\right)^{\frac{1}{2}}}\\
    &\leq&\nonumber C(h)+ (A_{1}+A_{2})^{1/2} (nh)^{-1/2} + 2C(h)\|K\|_1 +2 V(h)  +2\left(\mathbb{E}(T_{1}^{2})\right)^{\frac{1}{2}}+2\left(\mathbb{E}(T_{2}^{2})\right)^{\frac{1}{2}}\\
    &\leq&
    A_6 C(h) + A_7 ( 1 + \delta_n ) \frac{(\log n)^{\frac{1}{2}}}{(nh)^{\frac{1}{2}}} +2\left(\mathbb{E}(T_{1}^{2})\right)^{\frac{1}{2}}+2\left(\mathbb{E}(T_{2}^{2})\right)^{\frac{1}{2}}.
\end{eqnarray*}
The Proposition \ref{AAPropET_1^2yV_1}, whose result holds for a sufficiently large $n$, allows us to obtain that
$$2\left(\mathbb{E}(T_{1}^{2})\right)^{\frac{1}{2}}+2\left(\mathbb{E}(T_{2}^{2})\right)^{\frac{1}{2}}\le  A_8 \frac{(\log n)^\frac{1}{2}}{n^\frac{1}{2}}$$ 
For some constant $A_8$ depending on the kernel $K$, of $r_{sup}$, $Q$, $g_{inf}$, $\sigma$, $a$, and $\gamma$, this allows us to conclude the result.
\end{proof}

\begin{theorem} \label{T3:adapt}\textbf{(Adaptability)} We assume that hypotheses $(H_1)$, $(H_2)$, $(H_3)$, $(H_4)$, and $(H_5)$ are satisfied. We assume that $K$ is of order $m$. Then, for all $0 < \beta \leq m$ and $L > 0$, the estimator $\hat{r}_{\hat{h}}$ satisfies for $r \in \Sigma(\beta, L)$
\begin{equation*}
     \left(\mathbb{E}\left(r(x)-\hat{r}_{\hat{h}}(x)\right)^2\right)^{\frac{1}{2}} \leq C^*\left(\frac{n}{\log n}\right)^{-\frac{\beta}{2\beta +1}}
\end{equation*}
where $C^*$ is a positive constant depending on the kernel $K$, of $\beta$, $L$, $r_{sup}$, $Q$, $g_{inf}$, $g_{sup}$, $\sigma$, $a$, and $\gamma$.
\end{theorem}

\begin{remark}
The estimator $\hat{r}_{\hat{h}}$ converges at a rate of $\left(\frac{n}{\log n}\right)^{-\frac{\beta}{2\beta +1}}$ over the classes of H\"older regularity $\beta$, and the estimator does not depend on $\beta$. This rate is nearly optimal or minimax, except for the logarithmic multiplicative term. This additional logarithmic term is common in adaptive settings and also appears, for example, in \cite{lepski2014adaptive}. This convergence rate result in pointwise risk for adaptive settings is novel in the context of dependent data and generalizes the result of \cite{lepski2014adaptive} obtained for independent data and known explanatory variable density.
\end{remark}
 \begin{proof}
Let $h_{i} = e^{-i}$ for each $i \in \{0,1,2,\cdots,M\}$ and $h_{\beta}=\left(\frac{\log n}{n}\right)^{\frac{1}{2\beta +1}}$. For sufficiently large $n$, there exists $i \in \{0,1,2,\cdots,M\}$ such that,
\begin{equation}\label{AADesigualhBeta}
    h_{i+1}\leq h_{\beta} \leq h_i
\end{equation}
where $h_{i+1}=e^{-i-1}=\frac{e^{-i}}{e}=\frac{h_i}{e}$.

The first term inside the minimum in inequality \eqref{eq:oracle} is bounded. By Proposition \ref{AApro:sesgo}, the bias of the estimator in the bandwidth $h_{i+1}$ satisfies for $u \in B(x)$,
\begin{equation*}
    |K_{h_{i+1}}\ast r(u)-r(u)|\leq A_0 h_{i+1}^{\beta}  
\end{equation*}
Therefore, by using inequality \eqref{AADesigualhBeta}, we have
\begin{equation}\label{AACotaC(hi+1)}
    C(h_{i+1})=\max_{u\in B(x)}\left|K_{h_{i+1}}\ast r(u)-r(u)\right|\leq A_0 h_{i+1}^{\beta}\le 
A_0 h_{\beta}^{\beta}=A_0 \left(\frac{n}{\log n}\right)^{-\frac{\beta}{2\beta +1}}.
\end{equation}
Now we proceed to determine an upper bound for the second term inside the minimum in inequality \eqref{eq:oracle}. From inequality \eqref{AADesigualhBeta}, it follows that $\frac{\log n}{nh_i} \leq \frac{\log n}{nh_{\beta}}$, and furthermore, as $h_{i+1}=\frac{h_i}{e}$, it is established that,
$$\frac{\log n}{n h_{i+1}}=\frac{e \log n}{n h_i} \leq \frac{e \log n}{n h_{\beta}}$$
which implies, using the value of $h_{\beta}$,

\begin{equation}\label{AACotaSegundoTDO}
\left(\frac{\log n}{n h_{i+1}}\right)^{\frac{1}{2}}\leq \frac{e^{\frac{1}{2}} (\log n)^{\frac{1}{2}}}{n^\frac{1}{2}}\frac{1}{h_{\beta}^{\frac{1}{2}}}=
e^{\frac{1}{2}}\left(\frac{\log n}{n}\right)^{\frac{\beta}{2\beta +1}}.
\end{equation}

By inequalities \eqref{AACotaC(hi+1)} and \eqref{AACotaSegundoTDO}, we have that for $h_{i+1} \in \mathcal{H}$,

$$A_6 C(h_{i+1})+A_7 (1+\delta_n ) \frac{(\log(n))^\frac{1}{2}}{(n h_{i+1})^\frac{1}{2}} \leq \left(A_6 A_0 +A_7 (1+\delta_n ) e^{\frac{1}{2}}\right)\left(\frac{n}{\log n}\right)^{-\frac{\beta}{2\beta +1}}.$$
Due to the above and inequality \eqref{eq:oracle}, it follows that,

\begin{eqnarray} \label{AAAdap1}
   \nonumber
    \left( \mathbb{E}\left(r(x)-\hat{r}_{\hat{h}}(x)\right)^2\right)^{\frac{1}{2}} &\leq& \left( A_6 A_0 + A_7 e^\frac{1}{2} (1+\delta_n) \right) \left(\frac{n}{\log n}\right)^{-\frac{\beta}{2\beta +1}} + A_8\frac{(\log n)^\frac{1}{2}}{n^\frac{1}{2}}\\
\end{eqnarray}

Using the fact that $\left(\frac{\log n}{n}\right)^{\frac{1}{2}} \leq \left(\frac{\log n}{n}\right)^{\frac{\beta}{2\beta +1}}$, it follows that
\begin{equation*}
     \left(\mathbb{E}\left(r(x)-\hat{r}_{\hat{h}}(x)\right)^2\right)^{\frac{1}{2}} \leq C^*\left(\frac{n}{\log n}\right)^{-\frac{\beta}{2\beta +1}}
\end{equation*}
where $C^*$ depends on $A_0$, $A_6$, $A_7$, and $A_8$, for $n \geq 3$.
 \end{proof}

\section{Simulation study} \label{AAsimulacion}

In this section, the fitting of the proposed procedure to simulated data is presented. Empirical global and local errors are compared for different sample sizes. Additionally, the procedure used for calibrating the method is outlined. The goal is to estimate the regression function for a weakly dependent process with a known density function.

\subsection{Simulation framework} \label{Aj}

The regression model \(Y_i=r(X_i)+\varepsilon_i\) is proposed for \(i=1,\ldots,n\), where the variables \(\varepsilon_i\) are independent and identically distributed, with a common distribution of \(N(0,\sigma^2)\), and \(\sigma=0.1, 0.5,\) and \(1\). The regression function to be estimated is \(r(x)=0.7x+2e^{-10x^2}\), which, when restricted to the interval \([-c,c]\), satisfies hypothesis \((H_{4})\).
 
A sample path \(\{X_i\}_{i=1}^{n+2q}\) is generated from the process \(\mathbb{X}\) presented in Example \ref{AAProcesoXenL}, where $X_{i}=\left(G^{-1}\circ\Phi_{0,\frac{\rho^2}{1-\phi^2}}\right)(Z_{i})$ with $\rho=1$, $\phi=0.75$ and $c=2$. The \(X_i\) follow a truncated normal distribution, specifically with a probability density function given by \(g(x) = \frac{\phi_{0,1}(x)}{p}\mathbf{1}_{[-2,2]}(x)\), where \(p = \Phi_{0,1}(2) - \Phi_{0,1}(-2)\) and \(x \in \mathbb{R}\). As mentioned in Example \ref{AAProcesoXenL}, the process satisfies \(\mathbb{X}\in\mathcal{L}\), meaning it fulfills hypotheses \((H_{1})\), \((H_{2})\), and \((H_{3})\).

\begin{figure}[htp!] 
	\centering 
	\begin{subfigure}
	\centering
	\includegraphics[width=7.66cm]{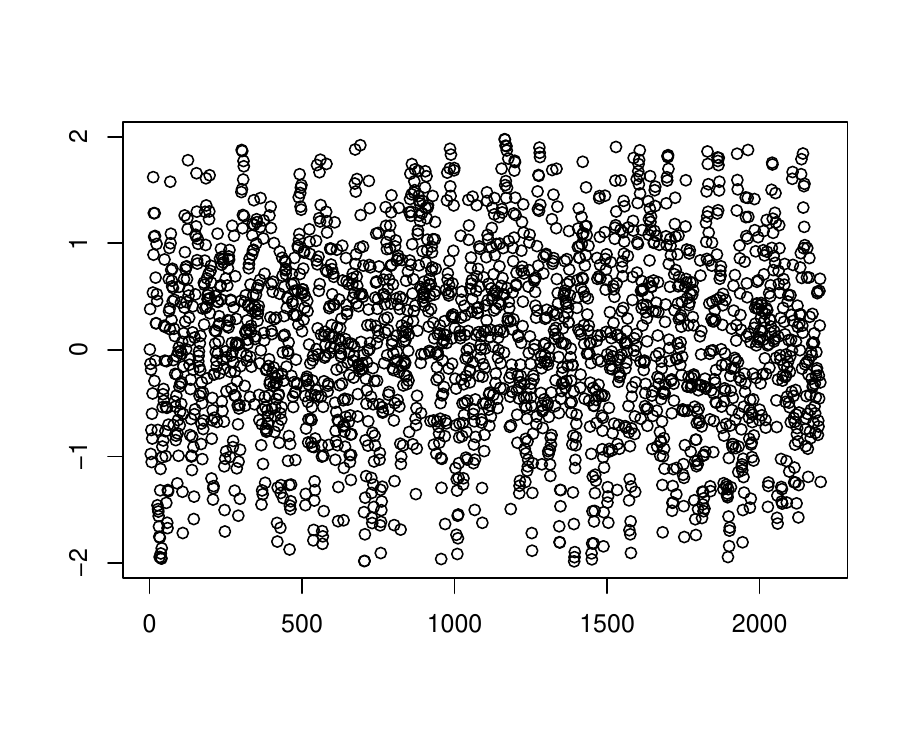}  
	\end{subfigure}
	\begin{subfigure}
	\centering
	\includegraphics[width=7.66cm]{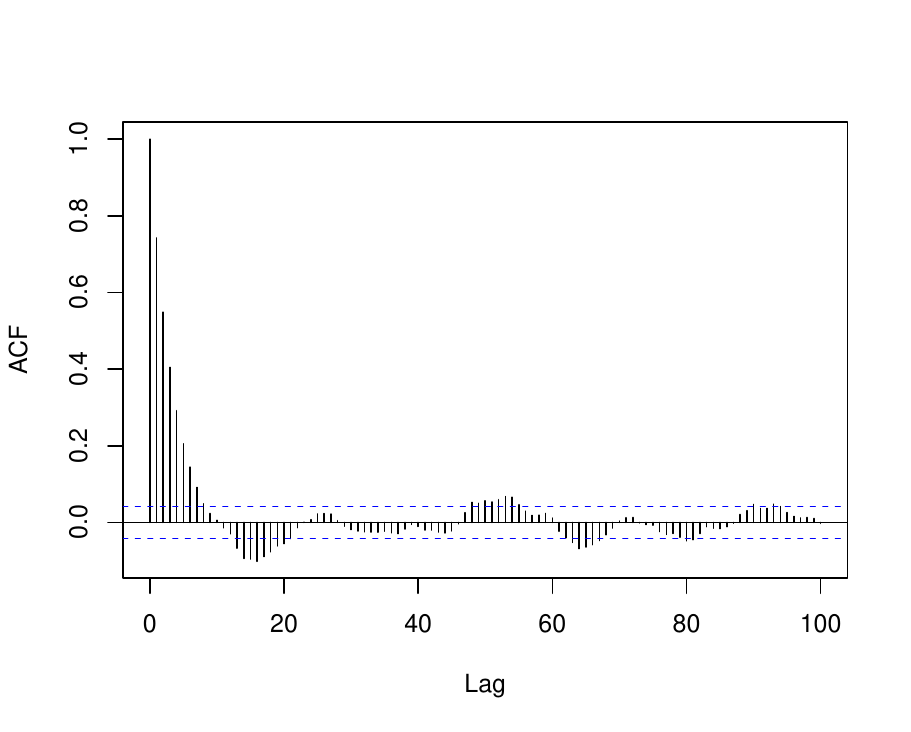}  
	\end{subfigure}
    \vspace{-0.5cm}
	\caption{Scatter plot of $\mathbb{X}$ (left) and sample autocorrelation function $\gamma_{X}$ (right).}\label{XyCovaX}
\end{figure}

Figure \ref{XyCovaX} shows the scatter plot of observations \(\{X_i\}_{i=1}^{n+2q}\) for \(n=2000\) and \(q=100\) (on the left) and the sample autocorrelation function \(\gamma_{X}(k) = \text{corr}(X_{i},X_{i+k})\) for \(k=1,\cdots,100\) (on the right). The exponential decay of correlations can be observed, a characteristic present in $\alpha$-weakly dependent processes.

The first \(n\) observations \(\{(X_i,Y_i)\}_{i=1}^{n}\) constitute the sample used in Section \ref{EstGL} for estimation using the GL method, and the last \(q\) observations \(\{(X_i,Y_i)\}_{i=n+q+1}^{n+2q}\) form the sample used for calibration by selecting an appropriate value of \(\gamma\) as described in Section \ref{Cali}. The idea of leaving a time gap of size \(q\) between the estimation sample and the calibration sample is to ensure that the coefficient of correlation \(\alpha_q(\mathbb{X})\) is sufficiently small, thus reducing the overfitting effect caused by dependence between samples. The control of the \(\alpha_q(\mathbb{X})\) correlation coefficient value is carried out through the hypothesis of exponential decay \((H_{3})\).

\subsection{GL Estimation} \label{EstGL}

The regression function \(r(\cdot)\) is estimated on the evenly spaced grid \(\{x_i\}_{i=1}^{s}\) within the interval \([-1,1]\), where \(x_{i}=-1+2\left(\frac{i-1}{s-1}\right)\) for \(i=1,\cdots,s\), and we consider \(s=21\). To apply the GL method, we consider the random sample \(\{(X_{i},Y_{i})\}_{i=1}^n\) as previously indicated, and determine the estimator \(\left\{ \hat{r}_{\hat{h}_{i}}(x_{i}) \right\}_{i=1}^s\), where

$$
\hat{r}_{\hat{h}_{i}}(x_{i})=\frac{1}{n}\sum_{k=1}^{n}Y_{k}K_{\hat{h}_i}(x_{i}-X_{k})g^{-1}(X_{k}),
$$
The kernel \(K(x) = \frac{1}{\sqrt{2\pi}}e^{-\frac{x^2}{2}}\) is used for all \(x \in \mathbb{R}\). Although it does not have compact support, it has good practical properties. Subsequently, \(K_{\hat{h}_i}(\cdot) = \frac{1}{\hat{h}_i}K\left(\frac{\cdot}{\hat{h}_i}\right)\) is taken, where each bandwidth width \(\hat{h}_i\) is selected from the bandwidth family \(\mathcal{H} = \left\{e^{-0.1j} : j=0,\cdots,\frac{[\log(n)]^{\frac{2}{3}}}{0.1}\right\}\), taking

\begin{equation}\label{hxi}
\hat{h}_{i}=\mbox{arg}\hspace{-0.05cm}\min_{h\in\mathcal{H}}\{A(h,x_{i})+V(h,x_{i})\}
\end{equation}
for each \(i=1,\cdots,s\), with $A(h,x_{i})= \max_{h'\in\mathcal{H}}\{|\hat{r}_{h,h'}(x_{i})-\hat{r}_{h'}(x_{i})|-V(h',x_{i})\}_+$, y $V(h,x_{i})=\sqrt{ 2\gamma A_{1}(i)}( \|K\|_1 + 1)( 1 + \delta_{n} ) \frac{\left(\log n\right)^\frac{1}{2}}{\left(nh\right)^\frac{1}{2}}$, where $\delta_n =(\log n)^{\frac{1}{5}}$, $A_{1}(i) = (g_{\inf}(i))^{-1}\|K\|_2^2((\hat{r}_{\sup}(i))^2+\hat{\sigma}^2)$, $\|K\|_{1}=1$, $\|K\|_{2}=\frac{1}{\sqrt{2\sqrt{\pi}}}$, $g_{inf}(i)= \min\left\{w
>0| w=g(X_{j})\mathbf{1}_{]x_{i}-0.5,x_{i}+0.5[}(X_j),\text{ for }j=1,\cdots,n \right\}$, $\hat{r}_{sup}(i)= \max\left\{|Y_j |\mathbf{1}_{]x_{i}-0.5,x_{i}+0.5[}(X_j)| j=1,\cdots,n \right\}$ and $\hat{\sigma}^2 =\frac{1}{n-1}\sum_{i=1}^n \left(Y_{i}-\tilde{r}_{\Bar{h}}(X_i)\right)^2$ where $\Tilde{r}_{\Bar{h}}(\cdot)$ is the Nadaraya-Watson estimator in the sample $\{(X_{i},Y_{i})\}_{i=1}^n$ with bandwidth $\Bar{h}$. Here, $\Bar{h}$ is the bandwidth resulting from the command $dpill(\cdot,\cdot)$ in $R$ corresponding to the $KernSmooth$ library, a plug-in method by Ruppert, Sheather, and Wand (1995).\\

The parameter \(\gamma\) is a calibration parameter of the method, and its value is determined according to the methodology proposed in Section \ref{Cali}. In Figure \ref{GrafEstima}, the regression function \(r\) and three GL estimators of the regression function for random samples \(\{(Y_{k},X_{k})\}_{k=1}^{n+2q}\) are shown, with \(n=1000\), \(2000\), and \(5000\) taking \(q=100\) and \(\sigma=0.5\).\\

\begin{figure}[h!]
    \begin{center}
        \includegraphics[width=9.9cm]{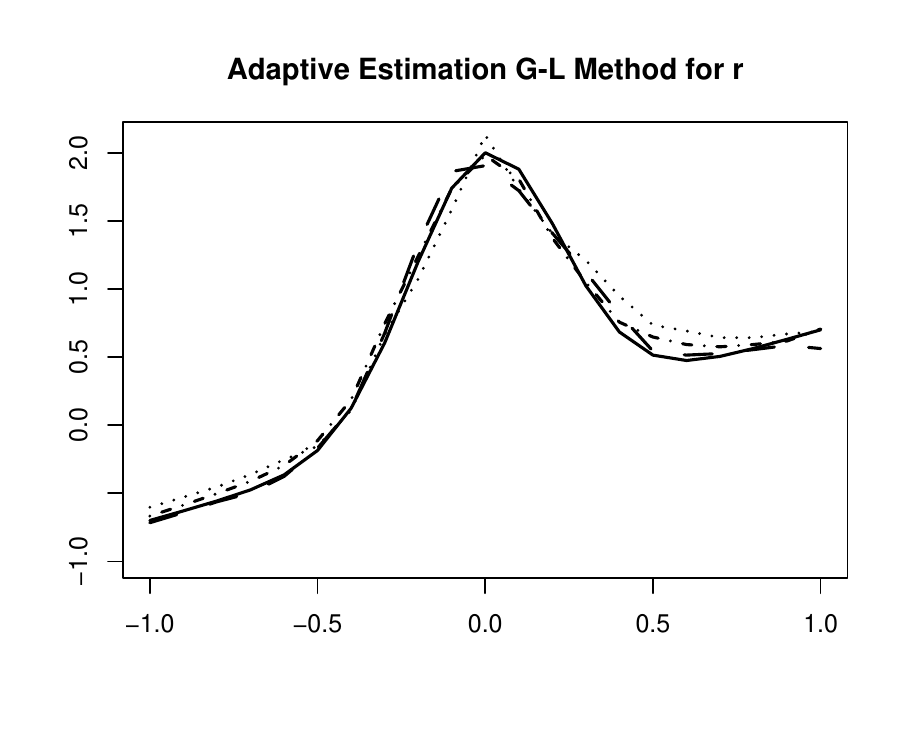}
    \end{center}
    \vspace{-1cm}
    \caption{Three GL estimators of the regression function \(r\) are displayed for random samples \(\{(Y_{k},X_{k})\}_{k=1}^{n+2q}\), with \(\sigma=0.5\) and \(q=100\). When \(n=1000\), the GL estimator is represented with dots; for \(n=2000\), with long dashes; and for \(n=5000\), with dots and short dashes. Additionally, the regression function \(r\) is depicted with a solid line.}
    \label{GrafEstima}
\end{figure}

\subsection{Calibration of the method for a sample.} \label{Cali}

As the estimation of the regression function \(r\) is performed at each point of the grid \(\{x_j\}_{j=1}^{s}\) within the interval \([-1,1]\), and for each \(x_j\), a bandwidth \(\hat{h}_j\) from the bandwidth family \(\mathcal{H}\) is selected, a vector of bandwidths \(\hat{h}=(\hat{h}_{1}, \ldots, \hat{h}_{s})\) is obtained. In general, the GL estimator of \(r\) over the grid \(\{x_i\}_{i=1}^{s}\) with the vector of optimal bandwidths \(\hat{h}=(\hat{h}_{1}, \ldots, \hat{h}_{s})\) is denoted by \(\left(\hat{r}_{\hat{h}_{1}}(x_{1}), \ldots, \hat{r}_{\hat{h}_{s}}(x_{s})\right)\), where each optimal bandwidth is obtained according to Equation \eqref{hxi} with \(V(h,x_{i})\) depending on the parameter \(\gamma>2\). In practice, to calibrate the method, \(\gamma\) is taken on a evenly spaced grid within a subinterval \(I\) of \(]0,1/10[\). For each \(\gamma\), the vector of optimal bandwidths depends on \(\gamma\), and this dependence is denoted as follows: \(\hat{h}^{\gamma}=(\hat{h}_{1}^{\gamma}, \ldots, \hat{h}_{s}^{\gamma})\). The choice of the interval \(I\) is made in such a way that the curve described by the estimators \(\hat{r}_{\hat{h}_{i}^{\gamma}}(x_{i})\), for \(i=1,\ldots,s\), transitions from being very irregular to smooth curves.\\ 

When calibrating the method for a sample \(\{(X_{i},Y_{i})\}_{i=1}^{n+2q}\), the first \(n\) data points are taken from the sample to construct the estimator \(\hat{r}_{h}(x)\). From the last \(q\) data points of the sample \(\{(X_{i},Y_{i})\}_{i=n+q+1}^{n+2q}\), a random grid is constructed within the interval \([-1,1]\) by selecting values that satisfy \(-1\leq X_i \leq 1\). The data are then sorted in ascending order with respect to the first coordinate, resulting in a sample denoted by lowercase letters \(\{(\tilde{x}_{i},\tilde{y}_{i})\}_{i=1}^p\) with \(0<p\leq 100\). The grid \(\{\tilde{x}_i\}_{i=1}^p\) is associated with the vector of optimal bandwidths \(\hat{h}^{\gamma}=(\hat{h}_{1}^{\gamma}, \ldots, \hat{h}_{p}^{\gamma})\) for each \(\gamma \in I\). In this way, GL estimators \(\hat{r}_{\hat{h}_{i}^{\gamma}}(\tilde{x}_{i})\) are obtained for each \(\gamma \in I\) over the grid \(\{\tilde{x}_i\}_{i=1}^p\) with the vector of optimal bandwidths \(\hat{h}^{\gamma}\).\\

The error is determined as \(Error(\gamma) = \sum_{i=1}^p d_{i}\left(\hat{r}_{\hat{h}_i^{\gamma}}(\tilde{x}_{i})-\tilde{y}_{i}\right)^2\) for each \(\gamma \in I\), where \(d_{1}=\frac{\tilde{x}_{1}+\tilde{x}_{2}}{2}+1\), \(d_{p}=1-\frac{\tilde{x}_{p-1}+\tilde{x}_{p}}{2}\), and \(d_{i}=\frac{\tilde{x}_{i+1}-\tilde{x}_{i-1}}{2}\) for \(i\in\{2,\ldots,p-1\}\). Finally, the method is calibrated by selecting the value of \(\gamma\) in \(I\) that minimizes \(Error(\gamma)\).\\

In practice, \(I=]0.00000005,0.05[\) was chosen, and over this interval, an evenly spaced grid \(\gamma_i = 0.00000005+(i-1)\frac{0.05-0.00000005}{20}\) for \(i=1,\ldots,21\) was considered.

\begin{figure}[h!] 
	\centering 
	\begin{subfigure}
	\centering
	\includegraphics[width=7.66cm]{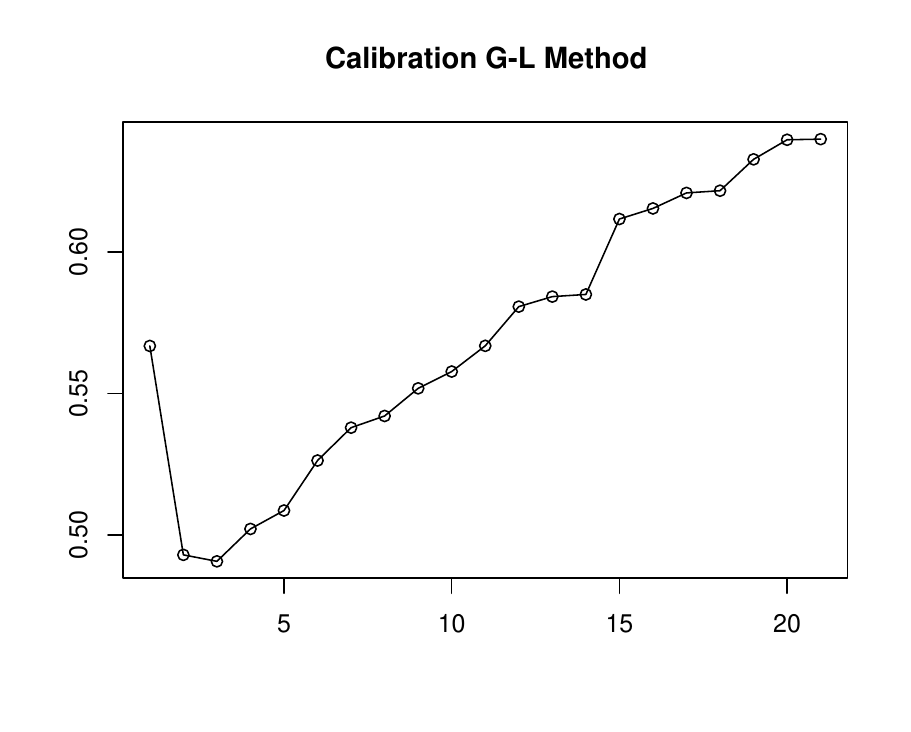}  
	\end{subfigure}
	\begin{subfigure}
	\centering
	\includegraphics[width=7.66cm]{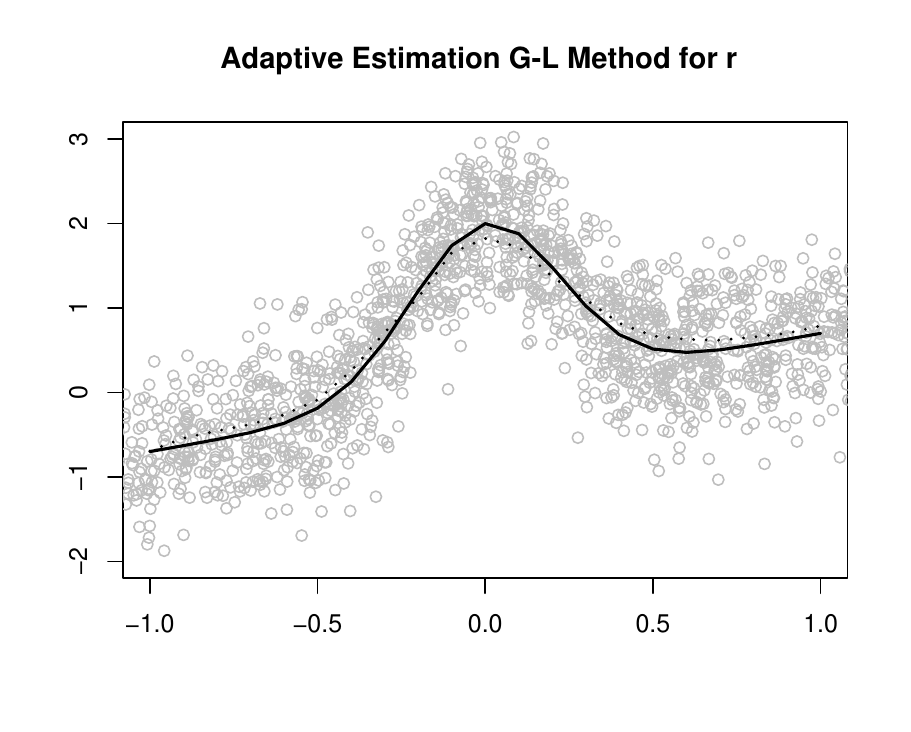}  
	\end{subfigure}
	\caption{On the left, a plot of \(\{Error(\gamma_i)\}_{i=1}^{21}\), and on the right, the point cloud \(\{(X_i ,Y_i )\}_{i=1}^n\) is shown. The regression function \(r\) is represented by a continuous line. Additionally, the plot of \(\{\hat{r}_{\hat{h}_i}(x_i)\}_{i=1}^{21}\) on the interval \([-1,1]\) is depicted with points, calibrated at \(\gamma_3\), where the left plot reaches its minimum. Both were taken with \(n=2000\), \(q=100\), and \(\sigma=0.5\).}\label{Calibracion}
\end{figure}

The Figure \ref{Calibracion}, in the left graph, shows that \(Error(\gamma)\) is minimized at \(\gamma_{3}\). Additionally, on the right side, the point cloud of the estimation sample is displayed. The regression function \(r\) is represented by a continuous line, and the points show the estimator \(\hat{r}_{\hat{h}_i}(x_{i})\) calibrated at \(\gamma_3\), evaluated on the deterministic grid of the interval \([-1,1]\), where \(x_{i}=-1+2\left(\frac{i-1}{s-1}\right)\) for \(i=1,\ldots,s\), with \(s=21\). Both were taken with \(n=2000\), \(q=100\), and \(\sigma=0.5\).

\subsection{Comparison of global and local empirical errors for different sample sizes and values of \texorpdfstring{\(\sigma\)}{sigma} .} \label{EstErrores}

To demonstrate the quality of the regression function estimator $r$, $N=500$ replicas of the sample $\{(X_{i},Y_{i})\}_{i=1}^{n+2q}$ are generated. Subsequently, $N$ calibrated GL estimators $\left\{\hat{r}^{(1)}(x_{i})\right\}_{i=1}^s,\cdots,\left\{\hat{r}^{(N)}(x_{i})\right\}_{i=1}^s$ are calculated, where $x_{i}=-1+2\frac{i-1}{s-1}$, $s=21$, and $\hat{r}^{(j)}(x_{i})=\hat{r}_{\hat{h}_i}(x_{i})$ with $\hat{h}_i$ obtained for the $j$-th replica using the GL method. The mean squared error ($MSE$) and the mean integrated squared error ($MISE$) are estimated as follows:

$$\hat{MSE}(x_{i})=\frac{1}{N}\sum_{j=1}^{N}\left[ \left(r(x_{i})-\hat{r}^{(j)}(x_{i})\right)^2\right],$$
for $i=1,\cdots,s$.

$$\hat{MISE}=\frac{1}{N}\sum_{j=1}^{N}I_j,$$
with $I_{j}=\sum_{i=1}^{s}d_i\left(r(x_{i})-\hat{r}^{(j)}(x_{i})\right)^2$, where $d_{1}=d_{s}=0.05$ and $d_i=0.1$ for $i=2,\cdots,s-1$.

\begin{figure}[htp!] 
	\centering 
	\begin{subfigure}
	\centering
	\includegraphics[width=7.66cm]{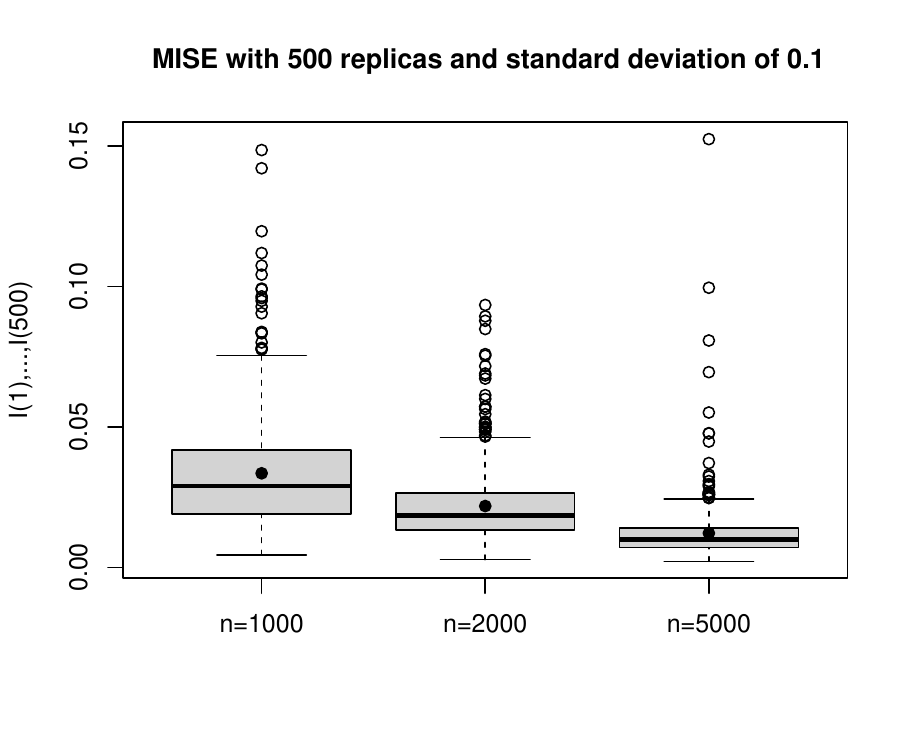}  
	\end{subfigure}
	\begin{subfigure}
	\centering
	\includegraphics[width=7.66cm]{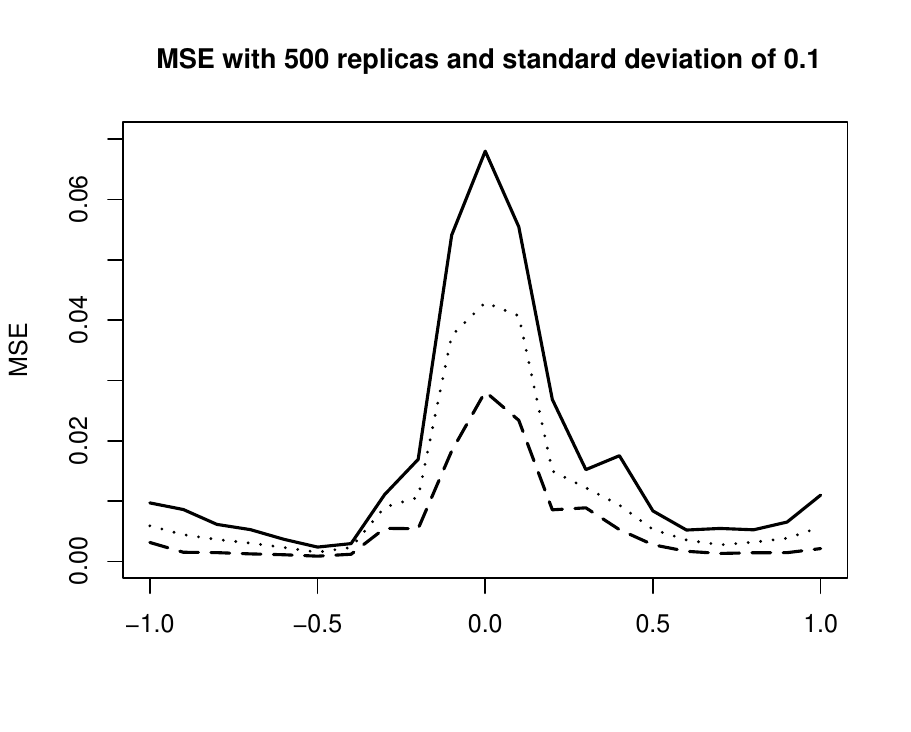}  
	\end{subfigure}
    \vspace{-0.5cm}
	\caption{On the left, three boxplots of $I_j$ for $j=1,\cdots,N$ are shown, with sample sizes $n=1000$, $2000$, and $5000$ from left to right, respectively. On each boxplot, there is a point representing $\hat{MISE}$. On the right, three representations of $\hat{MSE}(x_{i})$ for $i=1,\cdots,s$ are displayed with a line for the sample size $n=1000$, dots for $n=2000$, and dashes for $n=5000$. In both cases, $\sigma=0.1$.}\label{MSEyMISE01}
\end{figure}

In Figure \ref{MSEyMISE01}, three boxplots of $I_j$ are shown for $j=1,\cdots,N$, with sample sizes $n=1000$, $2000$, and $5000$ from left to right, respectively. Additionally, on each boxplot, there is a point representing $\hat{MISE}$ (on the left) and three representations of $\hat{MSE}(x_{i})$ for $i=1,\cdots,s$ with a line for sample size $n=1000$, dots for $n=2000$, and dashes for $n=5000$ (on the right). In both cases, $\sigma=0.1$. In Figures \ref{MSEyMISE05} and \ref{MSEyMISE1}, the same representations are made with $\sigma=0.5$ and $\sigma=1$, respectively.

Upon analyzing the graph corresponding to $\hat{MSE}$ with $\sigma=0.1$ in Figure \ref{MSEyMISE01} (right-hand side plots), it is observed that the $\hat{MSE}$ values are higher for the values of the domain of $r$ near where the maximum is reached and where there is a change in concavity. This pattern is observed for sample sizes $n=1000$, $2000$, and $5000$. Clearly, as the sample size increases, the $\hat{MSE}$ values decrease. Similar trends are observed when analyzing $\hat{MSE}$ for $\sigma=0.5$ and $\sigma=1$ in Figures \ref{MSEyMISE05} and \ref{MSEyMISE1}, respectively (right-hand side plots). Clearly, as the value of $\sigma$ increases, the $\hat{MSE}$ values also increase.

In the boxplots of the values of $I_{j}$ for $j=1,\cdots,N$ with $\sigma=0.1$ corresponding to Figure \ref{MSEyMISE01} (left-hand side plots), it is observed that as the values of $n$ increase from $1000$ to $2000$ and then to $5000$, the dispersion of $\{I_j\}_{j=1}^N$ decreases. The number of outliers also decreases, and the value of $\hat{MISE}$ approaches the median of $\{I_j\}_{j=1}^N$, although generally, the median is always lower than $\hat{MISE}$. In the boxplots of Figures \ref{MSEyMISE05} and \ref{MSEyMISE1} corresponding to $\sigma=0.5$ and $\sigma=1$, respectively, it is observed that as the values of $\sigma$ increase, the dispersion of $\{I_j\}_{j=1}^N$ also increases. The distance between $\hat{MISE}$ and the median of $\{I_j\}_{j=1}^N$ increases, and generally, the median is always lower than the value of $\hat{MISE}$.

\begin{figure}[htp!] 
	\centering 
	\begin{subfigure}
	\centering
	\includegraphics[width=7.66cm]{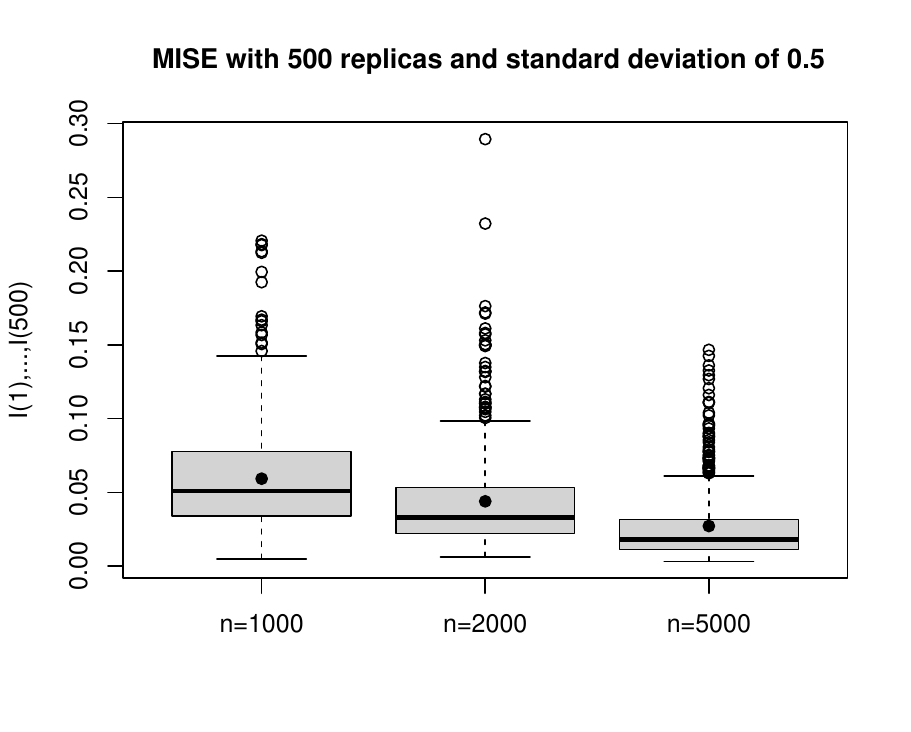}  
	\end{subfigure}
	\begin{subfigure}
	\centering
	\includegraphics[width=7.66cm]{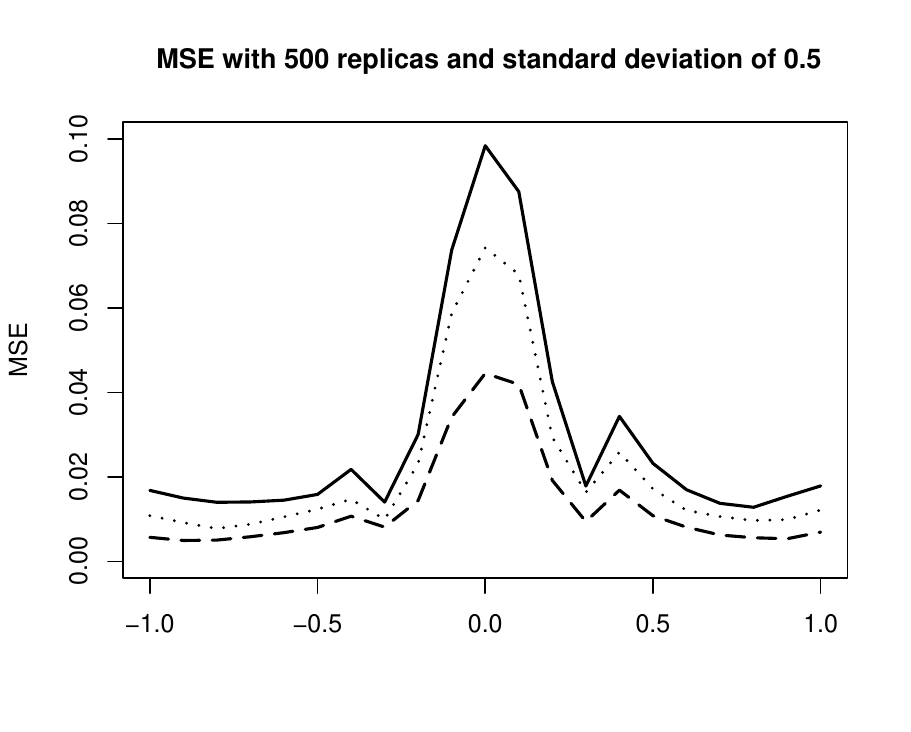}  
	\end{subfigure}
    \vspace{-0.5cm}
	\caption{On the left, three boxplots of $I_j$ for $j=1,\cdots,N$ are shown, with sample sizes $n=1000$, $2000$, and $5000$ from left to right, respectively. On each boxplot, there is a point representing $\hat{MISE}$, and on the right, three representations of $\hat{MSE}(x_{i})$ for $i=1,\cdots,s$ are displayed with a line for the sample size $n=1000$, dots for $n=2000$, and dashes for $n=5000$. In both cases, $\sigma=0.5$.}\label{MSEyMISE05}
\end{figure}

 \begin{figure}[htp!] 
	\centering 
	\begin{subfigure}
	\centering
	\includegraphics[width=7.66cm]{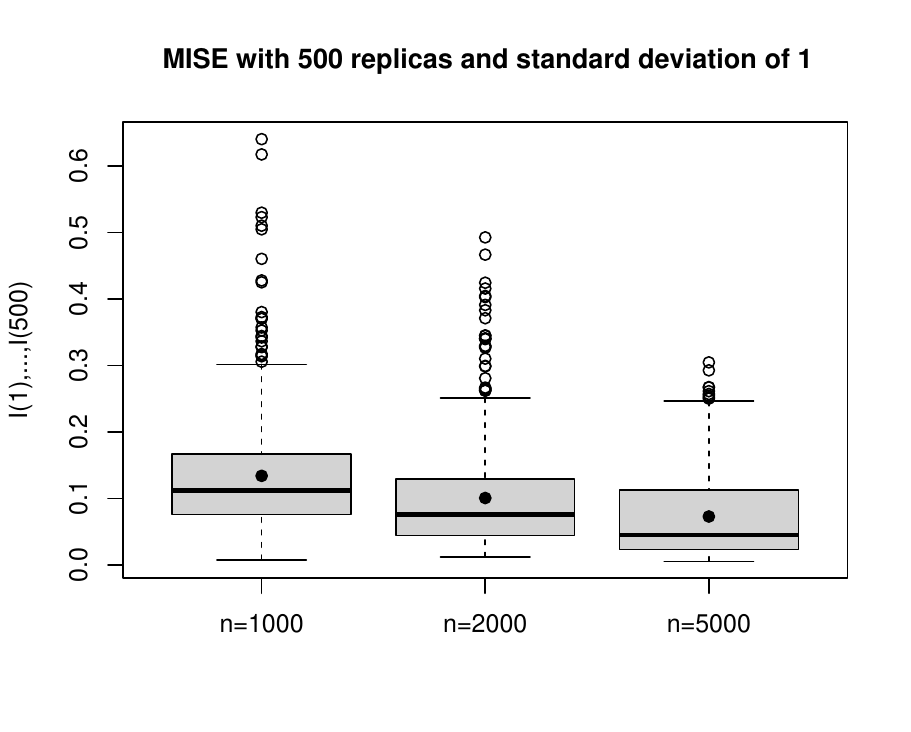}  
	\end{subfigure}
	\begin{subfigure}
	\centering
	\includegraphics[width=7.66cm]{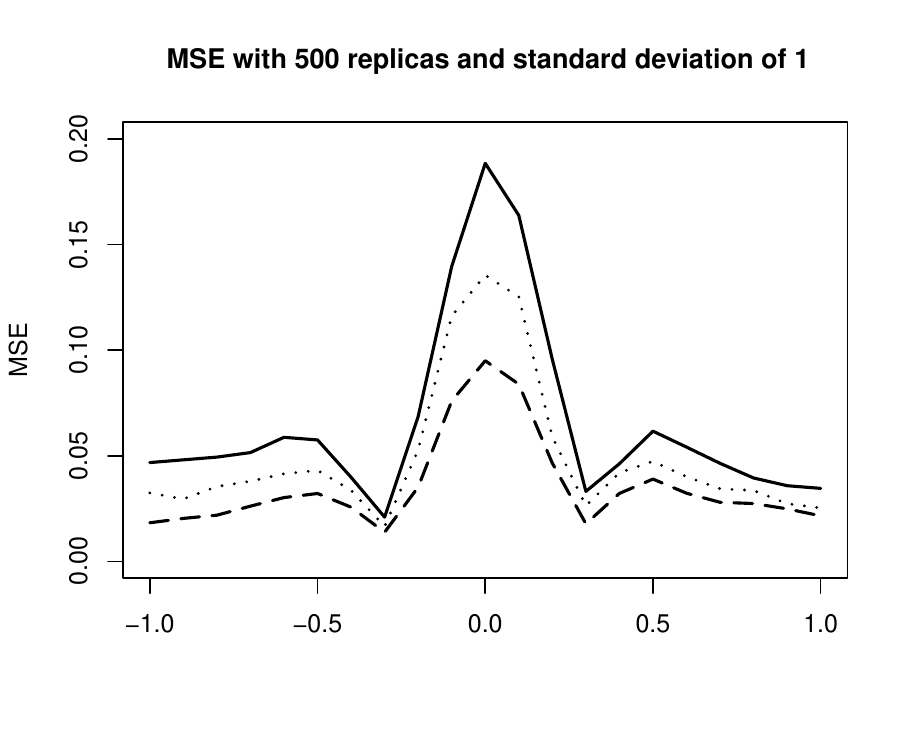}  
	\end{subfigure}
    \vspace{-0.5cm}
	\caption{On the left, three boxplots of $I_j$ for $j=1,\cdots,N$ are shown, with sample sizes $n=1000$, $2000$, and $5000$ from left to right, respectively. On each boxplot, there is a point representing $\hat{MISE}$, and on the right, three representations of $\hat{MSE}(x_{i})$ for $i=1,\cdots,s$ are displayed with a line for the sample size $n=1000$, dots for $n=2000$, and dashes for $n=5000$. In both cases, $\sigma=1$.}\label{MSEyMISE1}
\end{figure}

As observed in Table \ref{AAtab:MISE}, the value of $\hat{MISE}$ is inversely proportional to $n$ and directly proportional to $\sigma$.

\begin{table}[h!] 
\begin{center}
\begin{tabular}{lccc}
\hline
 & $\boldsymbol{n=1000}$ & $\boldsymbol{n=2000}$ & $\boldsymbol{n=5000}$\\
$\boldsymbol{\sigma=0.1}$ & 0.03351511 & 0.02186702 & 0.01219458\\
$\boldsymbol{\sigma=0.5}$ & 0.05932621 & 0.04399132 & 0.02723891\\
$\boldsymbol{\sigma=1}$ & 0.13400681 & 0.10064675 & 0.07281427\\
\hline  
\end{tabular}
\caption{Values of $\hat{MISE}$ for different values of $\sigma$ and sample sizes, known g case.}
\label{AAtab:MISE}
\end{center}
\end{table}

\appendix
\section{Technical results}\label{AARTecnicos}
\vspace{0.3cm}
\subsection{Constants}\label{AAConstantes}
We recall some values of constants given previously, and additionally, the following are constants that we will use in the development of the proofs.
\begin{align*}
&A_0 =\frac{L}{\mathit{l}!}\int_{\mathbb{R}}|u|^{\beta}|K(u)|du, \quad A_{1} = (g_{\inf})^{-1}\|K\|_2^2((r_{\sup})^2+\sigma^2)=B_{1}+B_{2},\\
&  A_2 = \frac{2A_3}{|\log a|} + 20A_4, \quad A_3 = (r_{\sup})^2 \|K\|_1^2 \left((g_{\inf})^{-2}Q+1\right),\\
&  A_4 =\frac{4}{1-a}(r_{\sup})^2(g_{\inf})^{-2}\|K\|_{\infty}^2,\quad   A_5 = B_{9} +2B_{10}, \\
& A_6 =1+2\|K\|_1,\quad  A_7 =\sqrt{A_{1} +A_2} +2\sqrt{ 2\gamma A_{1}}( \|K\|_1 + 1),\quad A_{8}=2(1+\|K\|_1^{2})A_{5},\\
& B=\frac{12(B_{6}\vee \sqrt{B_7})}{1 \wedge |\log a|}\left( \frac{2^{7}3B_{7}a^{-\frac{1}{3}}}{A_{1}(1 \wedge |\log a |)} \vee 1\right), \\
& B_{1} =\sigma^2 (g_{\inf})^{-1}\|K\|_2^2, \quad B_{2} =(r_{\sup})^2(g_{\inf})^{-1}\|K\|^2_2,\quad B_{3} =\left( r_{\sup} + \frac{2\sigma}{\sqrt{2\pi}}\right)\|K\|_1,\\
& B_{4} = \left( r_{\sup} + \frac{2\sigma}{\sqrt{2\pi}}\right)^2 (g_{\inf})^{-2}Q \|K\|_1^2, \quad B_{5} = B_4+B_3^2,\quad B_{6} = 2\|K\|_{\infty}(g_{\inf})^{-1},\\
& B_{7} = (4B_{5} B_{6} (\sigma+r_{\sup}))^{\frac{2}{3}},\quad B_{8} = \frac{2B_{5}}{|\log a|} +\frac{18B_{7}}{1-a^\frac{1}{3}},\\
& B_{9} =  (g_{\inf})^{-2}\|K\|_{\infty}^2 \left\{(8((r_{\sup})^4 +3\sigma^4)\right\}^{\frac{1}{2}}\left(\frac{2}{\sqrt{2\pi}}\right)^{\frac{1}{2}},\\
& 
B_{10} = \frac{2^4 3}{1-e^{-\frac{\gamma-2}{2}}}(A_{1} + B_{8}) + \frac{5!2^{18} 3 B^2}{1-e^{-\frac{\gamma-2}{2}}} (\sigma + r_{\sup})^{2},\\
& C_{*}=A_{0}^{2}+A_{1}+A_{2}  \quad \text{and} \quad C^{*}=A_{0}A_{6}+2e^{-\frac{1}{2}}A_{7}+A_{8}.
\end{align*}

\subsection{Statements of technical results.} \label{AAERT}

\begin{lemma} \label{AAL1}
Under the assumptions $H_1$, $H_2$, $H_4$, and $H_5$, it is obtained that:
\begin{enumerate}
    \item[i)] $\mathbb{E}\left[\left( K_h (x-X_i)\varepsilon_i g^{-1}(X_i) \right)^2\right] \leq \frac{B_{1}}{h}$, 
    \item[ii)] $\mathbb{E}\left[ \left( K_h (x-X_i)r(X_i)g^{-1}(X_i) \right)^2 \right] \leq \frac{B_{2}}{h}$, 
    \item[iii)] $\mathbb{E}\left[ |K_h(x-X_i)K_h(x-X_j)| \right] \leq Q\|K\|_1^2$,
    \item[iv)] $\mathbb{E}\left[ |K_h(x-X_i)r(X_i)g^{-1}(X_i )| \right] \leq r_{\sup} \|K\|_1$,
\end{enumerate}
where $B_{1} =\sigma^2 (g_{\inf})^{-1}\|K\|_2^2, $ and $ B_{2} =(r_{\sup})^2(g_{\inf})^{-1}\|K\|^2_2$.
\end{lemma}

\begin{lemma} \label{AAL2}
 For each $h \in \mathcal{H}$, $a \in ]0,1[$, and $n \geq 4$, it holds that:
 $$\frac{1}{h}a^{\left\{\frac{1}{h|\log a|(\log n)^\frac{1}{2}}\right\}} \leq \frac{10}{(\log n)^\frac{5}{2}}.$$
\end{lemma}

 \begin{proposition} \label{AAP1}
 Assuming that the $X_i$ satisfy hypothesis $(H_3)$. Let $k, u, v \in \mathbb{N}$, such that ${i_{1:u}} \in \mathbb{Z}^u$, ${j_{1:v}} \in \mathbb{Z}^v$, and $i_1 \leq \cdots \leq i_u < i_u + k \leq j_1 \leq \cdots \leq j_v$. If $G_u : \mathbb{R}^u \times \mathbb{R}^u \rightarrow \mathbb{R}$ and $G_v : \mathbb{R}^v \times \mathbb{R}^v \rightarrow \mathbb{R}$ are functions such that $\|G_u\|_{\infty} < \infty$ and $\|G_v\|_{\infty} < \infty$, then,
 $$|cov\left(G_u(X_{i_{1:u}},\varepsilon_{i_{1:u}} ),G_v (X_{j_{1:v}} ,\varepsilon_{j_{1:v}} )\right)|\leq \Psi (u,v,G_u ,G_v ) a^k$$
 where $\Psi(u,v,G_u , G_v )=4\|G_u \|_{\infty}\|G_v\|_{\infty}$, with $\|G_u\|_{\infty}=\sup_{(x,y)\in\mathbb{R}^u\times\mathbb{R}^u}|G_u (x,y)|$ and $\|G_v\|_{\infty}=\sup_{(x,y)\in\mathbb{R}^v\times\mathbb{R}^v}|G_v (x,y)|$.
 \end{proposition}

Next, we will work with specific functions $G_u$. More precisely, we consider, for $u \in \mathbb{N}$ and ${i_{1:u}} \in \mathbb{Z}^u$, $G_u : \mathbb{R}^u \times \mathbb{R}^u \rightarrow \mathbb{R}$ given by
 $$G_u (X_{i_{1:u}} ,\varepsilon_{i_{1:u}} )=\prod_{k=1}^u\left\{G(X_{i_k},\varepsilon_{i_k})-\mathbb{E}\left(G(X_{i_k},\varepsilon_{i_k})\right)\right\}$$
 where for $i \in \mathbb{N}$,
 $$G(X_{i},\varepsilon_{i}) =\frac{1}{n}(r(X_{i})+\varepsilon_{i})K_h (x-X_i)g^{-1}(X_i )\mathbf{1}_{\{|(r(X_{i})+\varepsilon_{i}|\leq M_n \}},$$
where $M_n = \sigma \log n + r_{\sup}$. The variables $G(X_{i}, \varepsilon_{i})$ satisfy the following results.
 
\begin{lemma} \label{AAL3}
    Under hypotheses $(H_1)$, $(H_2)$, $(H_4)$, and $(H_5)$, it holds that: $\forall i, k \in \mathbb{N}$
 \begin{enumerate}
     \item[i)] $|\mathbb{E}\left[G(X_i ,\varepsilon_i ) \right]| \leq \frac{B_{3}}{n}$,
     \item[ii)] $|\mathbb{E}\left[ G(X_i ,\varepsilon_i )G(X_{i+k} ,\varepsilon_{i+k}) \right]| \leq \frac{B_{4}}{n^2}$,
     \item[iii)] $|Cov\left( G(X_i ,\varepsilon_i ), G(X_{i+k} ,\varepsilon_{i+k}) \right)| \leq \frac{B_{5}}{n^2}$,
     \item[iv)] $|G(X_{i},\varepsilon_{i})| \leq \frac{B_{6}}{2} \frac{M_n}{nh}$
 \end{enumerate}
 where $$B_{3} =\left( r_{\sup} + \frac{2\sigma}{\sqrt{2\pi}}\right)\|K\|_1, \quad B_{4} = \left( r_{\sup} + \frac{2\sigma}{\sqrt{2\pi}}\right)^2 (g_{\inf})^{-2}Q \|K\|_1^2 $$ 
 $$B_{5} = \left( r_{\sup} + \frac{2\sigma}{\sqrt{2\pi}}\right)^2 \|K\|_1^2 \left((g_{\inf})^{-2}Q+1\right) \quad \text{and \ } B_{6} = 2\|K\|_{\infty}(g_{\inf})^{-1}.$$
 \end{lemma}

 \begin{proposition} \label{AAP2}
Let $u, v, k \in \mathbb{N}$, and if $(i_1, \ldots, i_u, j_1, \ldots, j_v) \in \mathbb{Z}^{u+v}$ such that $i_1 \leq \cdots \leq i_u < i_u + k \leq j_1 \leq \cdots \leq j_v$. Under hypotheses $(H_1)$, $(H_2)$, $(H_3)$, $(H_4)$, and $(H_5)$, we have, $\forall h \in \mathcal{H}$:
 $$\gamma_h (u,v)=|cov\left(G_u (X_{i_{1:u}} ,\varepsilon_{i_{1:u}}), G_v (X_{j_{1:v}} ,\varepsilon_{j_{1:v}})\right)| \leq \phi (u,v)(D_1(h))^{u+v-2}D_2 (h)  a^{\frac{k}{3}}$$
 where $\phi(u,v)= u+v+uv$,
 $$D_1 (h)=\frac{B_{6} M_n}{nh}, \quad D_2 (h)=\frac{B_{7}}{n^2 h}\quad \text{and \ } B_{7} = (4B_{5} B_{6} (\sigma+r_{\sup}))^{\frac{2}{3}}.$$
 \end{proposition}

 \begin{proposition} \label{AATDBernstein}
\textbf{Bernstein Inequality}. Consider the process $\mathbb{Z} = (Z_i)_{i \in \mathbb{N}}$, where $Z_i = G(X_i, \varepsilon_i) - \mathbb{E}[G(X_i, \varepsilon_i)]$, and $S_n = \sum_{i=1}^n Z_i$. Under hypotheses $(H_1)$, $(H_2)$, $(H_3)$, $(H_4)$, and $(H_5)$, for $h \in \mathcal{H}$ and $\forall t \geq 0$, we have:
\begin{equation} \label{AADBernstein}
    \mathbb{P}(|S_n | \geq t) \leq \exp \left(-\frac{\frac{t^2}{2}}{\mathcal{A}_n +\mathcal{B}_n^\frac{1}{3}t^\frac{5}{3}}\right)
\end{equation}
where $\mathcal{A}_n = \frac{A_{1}}{nh} + B_{8} \frac{(\log n)^{-\frac{1}{2}}}{nh}$ is an upper bound for $Var(S_{n})$, and $\mathcal{B}_n = B \frac{M_n}{nh}$, with
$$B_{8} = \frac{2B_{5}}{|\log a|} +\frac{18B_{7}}{1-a^\frac{1}{3}},\quad  B=\frac{12(B_{6}\vee \sqrt{B_7})}{1 \wedge |\log a|}\left( \frac{2^{7}3B_{7}a^{-\frac{1}{3}}}{A_{1}(1 \wedge |\log a |)} \vee 1\right).$$
\end{proposition}
 
Following the procedure outlined in Lemma 3 of \cite{bertin2017pointwise}, the Bernstein inequality can be reformulated as stated in the following Corollary.
 
 \begin{corollary} \label{AACDBernstein} Under the assumptions of Proposition \ref{AATDBernstein}, if $\lambda(u) = (2\mathcal{A}_n u)^\frac{1}{2} + \mathcal{B}_n (2u)^3$ for each $u \geq 0$ and we make the change of variable $t = \lambda(u)$ in the Bernstein inequality \eqref{AADBernstein}, then this inequality can be rewritten as follows:
 
 $$ \mathbb{P}(|S_n | \geq \lambda(u)) \leq \exp \left(-\frac{u}{2}\right)$$
 for each $u \geq 0$.
 \end{corollary}

\subsection{Proofs of technical results.} \label{AADRT}
\vspace{0.5cm}
\subsubsection{Proof of Lemma \ref{AAL1}.} 

i) We have
\begin{eqnarray*}
    \nonumber
    \mathbb{E}\left[\left(K_h (x-X_i)\varepsilon_i g^{-1}(X_i) \right)^2\right] &=& \mathbb{E}\left[K_h^2 (x-X_i )\varepsilon_i^2 g^{-2}(X_i) \right]\\
    \nonumber
    &=&
    \frac{\sigma^2}{h^2} \mathbb{E}\left[K^2 \left(\frac{x-X_i}{h}\right) g^{-2}(X_i)\right]\\
    \nonumber
    &=&
    \frac{\sigma^2}{h^2} \int_{\mathbb{R}}K^2 \left(\frac{x-t}{h}\right) g^{-1}(t)dt\\
    \nonumber
    &\leq&
    \frac{\sigma^2 (g_{\inf})^{-1}}{h^2} \int_{\mathbb{R}}K^2 (u)hdu=\frac{B_{1}}{h}.
\end{eqnarray*}
ii) We have
\begin{eqnarray*}
   \nonumber
   \mathbb{E}\left[ \left( K_h (x-X_i)r(X_i)g^{-1}(X_i) \right)^2 \right] &=&  \mathbb{E} \left[\frac{1}{h^2} K^2 \left(\frac{x-X}{h}\right)r^2(X )g^{-2}(X)\right]\\
   \nonumber
   &=&
   \frac{1}{h^2} \int_{\mathbb{R}} K^2 \left(\frac{x-t}{h}\right)r^2(t )g^{-1}(t)dt\\
   \nonumber
   &\leq&
   \frac{(r_{\sup})^2 (g_{\inf})^{-1}}{h^2} \int_{\mathbb{R}} K^2 (u)hdu= \frac{B_{2}}{h}.
\end{eqnarray*}
iii) We have
\begin{eqnarray*}
   \mathbb{E}\left[ |K_h(x-X_i)K_h(x-X_j)| \right] &=& \int\int |K_h(x-t) K_h (x-s)|g_{i,j}(t,s)dtds\\
   &\leq&
   \frac{1}{h^2}\int\int \left| K\left(\frac{x-t}{h}\right) K\left(\frac{x-s}{h}\right)\right|Qdtds\\  
   &\leq&
   \frac{Q}{h^2}\int\int |K(u)| |K(v)|h^2dudv= Q\|K\|_1^2
\end{eqnarray*}
iv) We have
\begin{eqnarray*}
   \mathbb{E}\left[ |K_h(x-X_i)r(X_i)g^{-1}(X_i)| \right]  &\leq& r_{\sup} \mathbb{E}\left[ |K_h(x-X_i)|g^{-1}(X_i) \right]\\
   &=&
    r_{\sup} \int | K_h(x-t)|dt\\
   &=&
 r_{\sup} \int \frac{1}{h}\left| K\left(\frac{x-t}{h}\right) \right|dt= r_{\sup}\|K\|_1
\end{eqnarray*}

\subsubsection{Proof of Lemma \ref{AAL2}.}
As $h \in \mathcal{H}$, we have $\frac{(\log n)^8}{n} \leq h < \frac{1}{(\log n)^2}$, which implies $\frac{\log n}{n} \leq h \leq \frac{1}{(\log n)^2}$. These inequalities are used to bound the expression $\frac{1}{h}a^{\left\{\frac{1}{h|\log a|(\log n)^\frac{1}{2}}\right\}}$.

\begin{eqnarray} \label{AACotaExpr}
   \nonumber
   \frac{1}{h}a^{\left\{\frac{1}{h|\log a|(\log n)^\frac{1}{2}}\right\}} &\leq&  \frac{n}{\log n}a^{\left\{\frac{1}{h} \frac{1}{|\log a|(\log n)^\frac{1}{2}}\right\}}\\
   \nonumber
   &\leq&
   \frac{n}{\log n}a^{\left\{(\log n)^2 \frac{1}{|\log a|(\log n)^\frac{1}{2}}\right\}}\\
    \nonumber
   &=&
   \frac{n}{\log n}a^{\frac{(\log n)^\frac{3}{2}}{|\log a|}}\\
   \nonumber
   &=&
   \frac{1}{\log n}a^{\log_a n}a^{\frac{(\log n)^\frac{3}{2}}{|\log a|}}\\
   \nonumber
   &=&
   \frac{1}{\log n}a^{\frac{(\log n)^\frac{3}{2}}{|\log a|} + \frac{\log n}{\log a}}\\
   &=&
   \frac{1}{\log n}a^{\frac{(\log n)^\frac{3}{2}}{|\log a|}\left\{1 - \frac{1}{(\log n)^\frac{1}{2}}\right\}}
\end{eqnarray}
If $n \geq 4$, then $-\frac{9}{10} < -\frac{1}{(\log n)^\frac{1}{2}} < 0$.\\

When considering $n \geq 4$ in inequality \eqref{AACotaExpr}, we have,
$$\frac{1}{h}a^{\left\{\frac{1}{h|\log a|(\log n)^\frac{1}{2}}\right\}} \leq \frac{1}{\log n}a^{\frac{(\log n)^\frac{3}{2}}{|\log a|}\left\{1 - \frac{9}{10}\right\}} = \frac{1}{\log n}a^{\frac{1}{10} \frac{(\log n)^\frac{3}{2}}{|\log a|}}$$

As $a\in ]0,1[$ and $e^{-x} \leq \frac{1}{x}, \forall x >0$ then,
$$a^{\frac{1}{10} \frac{(\log n)^\frac{3}{2}}{|\log a|}} \leq \left(e^{\log a}\right)^{\frac{1}{10} \frac{(\log n)^\frac{3}{2}}{|\log a|}} \leq e^{\left(-\frac{1}{10} \frac{(\log n)^\frac{3}{2}}{|\log a|}|\log a|\right)} \leq \frac{10}{(\log n)^\frac{3}{2}}$$

From the two previous inequalities, we conclude that,

$$\frac{1}{h}a^{\left\{\frac{1}{h|\log a|(\log n)^\frac{1}{2}}\right\}} \leq \frac{10}{(\log n)^\frac{5}{2}}$$
for $n\geq 4$.

\subsubsection{Proof of Proposition \ref{AAP1}.}
We have 
\begin{equation} \label{AAEqLema1}
cov(G_u (X_{i_{1:u}} ,\varepsilon_{i_{1:u}}),G_v (X_{j_{1:v}} ,\varepsilon_{j_{1:v}})) = \mathbb{E}\left[cov(G_u (X_{i_{1:u}} ,\varepsilon_{i_{1:u}}),G_v (X_{j_{1:v}} ,\varepsilon_{j_{1:v}})|\varepsilon_{i_{1:u}} \varepsilon_{j_{1:v}})\right]. 
\end{equation}
Since the $\varepsilon_i$ are independent of the $X_i$, using equation \eqref{AACoefDevil}, 
\begin{eqnarray*}
|cov(G_u (X_{i_{1:u}} ,\varepsilon_{i_{1:u}}),G_v (X_{j_{1:v}} ,\varepsilon_{j_{1:v}})|\varepsilon_{i_{1:u}} \varepsilon_{j_{1:v}})| &\leq& 4\|G_u (\cdot,\varepsilon_{i_{1:u}} )\|_{\infty}\|G_v (\cdot,\varepsilon_{j_{1:v}} )\|_{\infty}\alpha_k (\mathbb{X})\\
&\leq&
4\|G_u \|_{\infty}\|G_v \|_{\infty}\alpha_k (\mathbb{X}).
\end{eqnarray*}
By hypothesis $(H3)$, the preceding inequality, and equation \eqref{AAEqLema1}, it follows that
\begin{eqnarray*}
|cov(G_u (X_{i_{1:u}} ,\varepsilon_{i_{1:u}}),G_v (X_{j_{1:v}} ,\varepsilon_{j_{1:v}}))| &\leq&4\|G_u \|_{\infty}\|G_v \|_{\infty}a^k.
\end{eqnarray*}

\subsubsection{Proof of Lemma \ref{AAL3}.}
i) We have
 \begin{eqnarray*}
 |\mathbb{E}(G(X_i,\varepsilon_i))| &\leq& \mathbb{E}\left\{\left|\frac{1}{n}(r(X_i)+\varepsilon_i)K_h (x-X_i)g^{-1}(X_i)\mathbf{1}_{\{|r(X_i)+\varepsilon_i|\leq M_n \}}\right|\right\}\\
  &\leq& 
  \frac{1}{n}\mathbb{E}\left\{|r(X_i)+\varepsilon_i||K_h (x-X_i)g^{-1}(X_i)|\right\}\\
  &\leq& 
  \frac{1}{n}\mathbb{E}\left\{|r(X_i)K_h (x-X_i)g^{-1}(X_i)|\right\}+\frac{1}{n}\mathbb{E}\left\{|\varepsilon_i K_h (x-X_i)g^{-1}(X_i)|\right\}\\
  &\leq& 
  \frac{1}{n}r_{\sup}\|K \|_1 + \frac{1}{n}\frac{2\sigma}{\sqrt{2\pi}}\int_{\mathbb{R}}|K_h (x-t)|dt= \frac{B_{3}}{n}
 \end{eqnarray*}
by inequality $iv)$ of Lemma \ref{AAL1}.\\

ii) It is denoted by $A(X_i)=\frac{1}{n}K_h (x-X_i)g^{-1}(X_i)$, whereby,
\begin{eqnarray*}    \lefteqn{\left|\mathbb{E}\left[G(X_i,\varepsilon_i)G(X_{i+k},\varepsilon_{i+k})\right]\right| }\\
&\leq& \mathbb{E}\left[\left|Y_i A(X_i)\mathbf{1}_{\{|Y_i|\leq M_n\}}Y_{i+k}A(X_{i+k})\mathbf{1}_{\{|Y_{i+k}|\leq M_n\}}\right|\right]\\
    &\leq& 
\mathbb{E}\left[|r(X_i)+\varepsilon_i||A(X_i)||r(X_{i+k})+\varepsilon_{i+k}||A(X_{i+k})|\right]\\
     &\leq& 
\mathbb{E}\left[|r(X_i)r(X_{i+k})A(X_i)A(X_{i+k})|\right] + \mathbb{E}\left[|r(X_i)\varepsilon_{i+k}A(X_i)A(X_{i+k})|\right]\\
    &+&
    \mathbb{E}\left[|\varepsilon_i r(X_{i+k})A(X_i)A(X_{i+k})|\right] + \mathbb{E}\left[|\varepsilon_i\varepsilon_{i+k}A(X_i)A(X_{i+k})|\right]\\
     &\leq& 
    (r_{\sup})^2\mathbb{E}\left[|A(X_i)A(X_{i+k})|\right] + r_{\sup}\mathbb{E}[|\varepsilon_{i+k}|]\mathbb{E}\left[|A(X_i)A(X_{i+k})|\right]\\
    &+&
    r_{\sup}\mathbb{E}[|\varepsilon_i|]\mathbb{E}\left[|A(X_i)A(X_{i+k})|\right] + \mathbb{E}[|\varepsilon_i|]\mathbb{E}[|\varepsilon_{i+k}|]\mathbb{E}\left[|A(X_i)A(X_{i+k})|\right]\\
     &=& 
    \left\{(r_{\sup})^2 + 2r_{\sup}\frac{2\sigma}{\sqrt{2\pi}}+\left(\frac{2\sigma}{\sqrt{2\pi}}\right)^2\right\}\mathbb{E}\left[|A(X_i)A(X_{i+k})|\right]\\
    &\leq&
     \left(r_{\sup} + \frac{2\sigma}{\sqrt{2\pi}}\right)^2 (g_{\inf})^{-2}\frac{1}{n^2}\mathbb{E}\left[|K_h(x-X_i)K_h(x-X_{i+k})|\right]\\
      &\leq&
     \left(r_{\sup} + \frac{2\sigma}{\sqrt{2\pi}}\right)^2 (g_{\inf})^{-2}Q\|K\|_1^2\frac{1}{n^2}\\
     &=&
     \frac{B_{4}}{n^2}.
\end{eqnarray*}
iii) Now using (i) and (ii), we have
\begin{eqnarray*}
\lefteqn{\left|Cov\left( G(X_i ,\varepsilon_i ), G(X_{i+k} ,\varepsilon_{i+k}) \right)\right|}\\ 
&\leq&
    |\mathbb{E}[G(X_i,\varepsilon_i)G(X_{i+k},\varepsilon_{i+k})]|+|\mathbb{E}[G(X_i,\varepsilon_i)]||\mathbb{E}[G(X_{i+k},\varepsilon_{i+k})]|\\
    &\leq&
    \frac{B_{4}}{n^2}+\frac{B_{3}}{n}\frac{B_{3}}{n}=
    \frac{B_{5}}{n^2}.
 \end{eqnarray*}
iv) The inequality (iv) is immediate using hypotheses $(H_1)$ and $(H_5)$.

\subsubsection{Proof of Proposition \ref{AAP2}.}

To prove this Proposition, two different bounds for the term

 $$\gamma_h (u,v)=\left|cov\left(G_u (X_{i_{1:u}} ,\varepsilon_{i_{1:u}}), G_v (X_{j_{1:v}},\varepsilon_{j_{1:v}})\right)\right|$$
are determined.
The first bound is obtained through direct calculation, while the second bound is derived from the dependency structure of the observations. Throughout the proof, $l=u+v$ is denoted.\\
 
\textbf{Direct Bound.} The proof of this bound consists of two steps. First, it is assumed that $l=2$, and then the case $l \geq 3$ is considered.\\
 
For $l=2$, we have $u=v=1$, $i_{1:1}=i_1$, and $j_{1:1}=j_1$. Therefore, using inequality $(iii)$ of Lemma \ref{AAL3}, we get
 \begin{equation}  \label{AACotaGammahl=2}
    \gamma_h (1,1) =
    |cov(G(X_{i_1},\varepsilon_{i_1}),G(X_{j_1},\varepsilon_{j_1}))|\leq 
    \frac{B_{5}}{n^2}.
 \end{equation}
 
Now, assume that $l \geq 3$. Without loss of generality, we can take $u \geq 2$ and $v \geq 1$. We have
 \begin{equation} \label{AACotaGammah0l>2}
     \gamma_h (u,v) = \left|cov\left(G_u (X_{i_{1:u}} ,\varepsilon_{i_{1:u}}),G_v (X_{j_{1:v}} ,\varepsilon_{j_{1:v}})\right)\right| \leq \Gamma_1+\Gamma_2
 \end{equation}
where
\begin{itemize}
    \item $\Gamma_1=\left| \mathbb{E}\left[ \prod_{k=1}^u\{ G(X_{i_k},\varepsilon_{i_k})-\mathbb{E}(G(X_{i_k},\varepsilon_{i_k}))\} \prod_{m=1}^v \{ G(X_{j_m},\varepsilon_{j_m})-\mathbb{E}(G(X_{j_m},\varepsilon_{j_m}))\} \right] \right|$
    \item $\Gamma_2=\left| \mathbb{E}\left[ \prod_{k=1}^u G(X_{i_k},\varepsilon_{i_k})-\mathbb{E}(G(X_{i_k},\varepsilon_{i_k})) \right] \mathbb{E}\left[ \prod_{m=1}^v G(X_{j_m},\varepsilon_{j_m})-\mathbb{E}(G(X_{j_m},\varepsilon_{j_m})) \right] \right|$
\end{itemize}
Using the fact that $|G(X_{i},\varepsilon_{i})-\mathbb{E}(G(X_{i},\varepsilon_{i}))| \leq B_{6} \frac{M_n}{nh}$ (see $(iv)$ Lemma \ref{AAL3}) and inequality \eqref{AACotaGammahl=2}, we have,
\begin{eqnarray*}
\Gamma_1 &\leq& \left(B_{6}\frac{M_n}{nh}\right)^{u+v-2}|\mathbb{E}\left\{[G(X_{i_1},\varepsilon_{i_1})-\mathbb{E}(G(X_{i_1},\varepsilon_{i_1}))][G(X_{i_2},\varepsilon_{i_2})-\mathbb{E}(G(X_{i_2},\varepsilon_{i_2}))]\right\}|\\
&\leq&
\left(B_{6}\frac{M_n}{nh}\right)^{l-2}\frac{B_{5}}{n^2}
\end{eqnarray*}

Analogously
\begin{equation*}
\Gamma_2 \leq \left(B_{6}\frac{M_n}{nh}\right)^{l-2}\frac{B_{5}}{n^2}
\end{equation*}

By inequality \eqref{AACotaGammah0l>2} and the bounds obtained for $\Gamma_1$ and $\Gamma_2$, we have $\forall l \geq 2$ that

\begin{equation} \label{AACotaDirecGamma} 
     \gamma_h (u,v) \leq 2\left(B_{6}\frac{M_n}{nh}\right)^{l-2}\frac{B_{5}}{n^2}.
 \end{equation}

\textbf{Structural Bound.} Another bound for $\gamma_h(u,v)$ will be provided using the dependency structure.\\

As $|G(X_{i_k},\varepsilon_{i_k})-\mathbb{E}(G(X_{i_k},\varepsilon_{i_k})) | \leq B_{6}\frac{M_n}{nh}$ for each $k=1,\cdots, u$, then

\begin{equation*}
|G_u (X_{i_{1:u}} ,\varepsilon_{i_{1:u}})| = \prod_{k=1}^u |G(X_{i_k},\varepsilon_{i_k})-\mathbb{E}(G(X_{i_k},\varepsilon_{i_k}))|\leq
\left(B_{6} \frac{M_n}{nh}\right)^u.
\end{equation*}

Therefore, $\|G_u\|_{\infty} \leq \left(B_{6} \frac{M_n}{nh}\right)^u$, and similarly, $\|G_v\|_{\infty} \leq \left(B_{6} \frac{M_n}{nh}\right)^v$. Due to this and Proposition \ref{AAP1}, we have

\begin{eqnarray} \label{AACotaEstrucGammah0}
\nonumber
\gamma_h (u,v) 
&\leq&
4\left(B_{6} \frac{M_n}{nh}\right)^l a^k\\
&=&
\left(B_{6} \frac{M_n}{nh}\right)^{l-2} 4 B_{6}^2 \frac{M_n^2}{n^2 h^2} a^k
\end{eqnarray}
Since $M_n = \sigma \log n + r_{\sup}$ and $h\leq \frac{1}{(\log n)^2}$, then $M_n^2 \leq (\sigma+r_{\sup})^2 (\log n)^2 \leq (\sigma+r_{\sup})^2\frac{1}{h}$. Due to this and inequality \eqref{AACotaEstrucGammah0}, we obtain the structural bound,

\begin{eqnarray} \label{AACotaEstrucGammah}
\gamma_h (u,v)  &\leq& \left( B_{6} \frac{M_n}{nh} \right)^{l-2} \frac{4B_{6}^2 (\sigma+r_{\sup})^2}{n^2 h^3} a^k
\end{eqnarray}

\textbf{Combination of the two bounds:} Now we combine \eqref{AACotaDirecGamma} and \eqref{AACotaEstrucGammah} as follows
\begin{eqnarray*} 
\gamma_h (u,v)  &=& \left(\gamma_h (u,v)\right)^{\frac{2}{3}}\left(\gamma_h (u,v)\right)^{\frac{1}{3}}\\
&\leq&
\left(2\left(B_{6}\frac{Mn}{nh}\right)^{l-2}\frac{B_{5}}{n^2}\right)^{\frac{2}{3}}\left(\left( B_{6} \frac{M_n}{nh} \right)^{l-2} \frac{4B_{6}^2 (\sigma+r_{\sup})^2}{n^2 h^3} a^k\right)^{\frac{1}{3}}\\
&=&
\left( B_{6} \frac{M_n}{nh} \right)^{l-2} (2B_{5})^{\frac{2}{3}}(2B_{6} (\sigma+r_{\sup}))^{\frac{2}{3}} \frac{1}{n^2 h} {a}^{\frac{k}{3}}\\
&=&
\left( B_{6} \frac{M_n}{nh} \right)^{l-2} \frac{B_{7}}{n^2 h} {a}^{\frac{k}{3}}\\
&\leq&
\left( D_1 (h) \right)^{l-2} D_2 (h) \phi (u,v) {a}^{\frac{k}{3}}.
\end{eqnarray*}

\subsubsection{Proof of Proposition \ref{AATDBernstein}.}

We will use Theorem 1 in \cite{doukhan2007probability}. This theorem states that if there exist constants $D_1, D_2, L_1, L_2 < \infty$ and a non-increasing sequence of real coefficients $(\rho_k)_{k\in\mathbb{N}}$ such that for every $u, v \in \mathbb{N}$, $i_{1:u}=(i_1,i_2,\cdots,i_u)\in\mathbb{N}^u$ and  $j_{1:v}=(j_1,j_2,\cdots,j_v)\in\mathbb{N}^v$ with $i_1\leq i_2 \leq \cdots \leq i_u < i_u +k \leq j_1 \leq j_2 \leq \cdots \leq j_v$, the following inequalities are satisfied:
 \begin{equation}\label{AACovDB}
   \left|cov\left(\prod_{s=1}^u Z_{i_s} , \prod_{s=1}^v Z_{j_s}\right)\right| \leq D_2 D_1^{u+v-2}\phi (u,v)\rho_k
\end{equation}
\begin{equation} \label{AASerieDB}
   \sum_{s=0}^\infty (s+1)^r \rho_s \leq L_1 L_2^r r!  , \  \forall r\in\mathbb{N},
\end{equation}
and
\begin{equation} \label{AAMomenDB}
    \mathbb{E}\left|Z_i\right|^r \leq D_1^r ,  \  \forall r\in\mathbb{N},
\end{equation}
then for all $t\geq 0$
\begin{equation*} \label{AADoukDB}
    \mathbb{P}(S_{n}\geq t) \leq \exp{\left( -\frac{\frac{t^{2}}{2}}{\mathcal{A}_{n}+(\mathcal{B}^*_{n})^{\frac{1}{3}}t^{\frac{5}{3}}} \right)}
\end{equation*}
where $\mathcal{A}_n$ is an upper bound for the variance of $S_{n}$ and
\begin{equation}\label{eq:Bn}
\mathcal{B}^*_n = 2(\sqrt{D_2} \vee D_1){L}_2\left(\frac{2^5 nD_2{L}_1}{\mathcal{A}_n}\vee 1\right).
\end{equation}

Next, we will verify that the three inequalities \eqref{AACovDB}, \eqref{AASerieDB}, and \eqref{AAMomenDB} hold. We will demonstrate that $Var(S_n)\le \mathcal{A}_n$ and $\mathcal{B}^*_n\le\mathcal{B}_n$. This implies the result of the proposition.\\

First, using Proposition \ref{AAP2}, we verify that inequality \eqref{AACovDB} is satisfied with $D_1=D_1(h)$, $D_2= D_2(h)$, and the non-increasing sequence $(\rho_k)_{k\in\mathbb{N}}$ where $\rho_k =a^{\frac{k}{3}}$ for $k\in\mathbb{N}$.\\

Now we will show that inequality \eqref{AASerieDB} holds. Let $r\in\mathbb{N}$. We have

\begin{align} \label{AASerieDB_01}
    \sum_{s=0}^\infty(s+1)^r \rho_s &=\sum_{s=0}^\infty (s+1)^r a^\frac{s}{3} = \sum_{s=0}^\infty \int_s^{s+1} (s+1)^r a^\frac{s}{3}dx\nonumber\\
     &\leq \sum_{s=0}^\infty \int_s^{s+1} (x+1)^{r}a^{\frac{x}{3}}a^{-\frac{1}{3}}dx = a^{-\frac{1}{3}}\int_{0}^\infty (x+1)^{r}a^{\frac{x}{3}}dx,
\end{align}
where we use that if $s \leq x \leq s+1$, then $a^{\frac{s}{3}} \leq a^{\frac{x}{3}}a^{-\frac{1}{3}}$ and $(s+1)^r \leq (x+1)^r$.

Applying integration by parts $r$ times, we have that.
\begin{eqnarray*}
\lefteqn{\int_0^\infty (x+1)^r a^\frac{x}{3}dx }\\
     &=&
     -\frac{3}{\log a}+r\left(-\frac{3}{\log a}\right)^2 + r(r-1)\left(-\frac{3}{\log a}\right)^3 + \cdots +r!\left(-\frac{3}{\log a}\right)^{r+1}\\
     &\leq&
     3^{r+1} r! \sum_{j=1}^{r+1}\left(\frac{1}{|\log a|}\right)^j
\end{eqnarray*}

Due to this last inequality and inequality \eqref{AASerieDB_01}, we have that,

\begin{eqnarray*}
    \sum_{s=0}^\infty(s+1)^r \rho_s &\leq& a^{-\frac{1}{3}}3^{r+1} r! \sum_{j=0}^{r+1}\left(\frac{1}{| \log a|}\right)^j\\
    &=&
    a^{-\frac{1}{3}}3^{r+1} r! \sum_{j=0}^{r+1}\frac{1}{| \log a|^j} \frac{2^j}{2^j}\\
    &\leq&
    a^{-\frac{1}{3}}3^{r+1} r! 2^{r+1}\sum_{j=0}^{r+1}\frac{1}{| \log a|^j} \frac{1}{2^j}\\
    &\leq&
    a^{-\frac{1}{3}}6^{r+1} r! \frac{1}{(1 \wedge|\log a|)^{r+1}} \sum_{j=0}^{r+1} \frac{1}{2^j}\\
    &\leq&
    a^{-\frac{1}{3}}\left(\frac{6}{1 \wedge|\log a|}\right)^{r+1}r!2
    =
    L_1 L_2^r r!
\end{eqnarray*}
where $L_1=\frac{12a^{-\frac{1}{3}}}{1 \wedge|\log a|}$ and $L_2 =\frac{6}{1 \wedge|\log a|}$, thus confirming the bound stated in inequality \eqref{AASerieDB}.\\

Inequality \eqref{AAMomenDB} is a consequence of (iv) of Lemma \ref{AAL3}.\\

Now, the bound $\mathcal{A}_{n}$ for the variance $S_{n}$ is determined. Using the fact that $\mathbb{E}[S_n]=0$, we have
\begin{equation} \label{AAVarSnDB0}
    \nonumber
    Var[S_n] =
    \mathbb{E}\left[\left(\sum_{i=1}^n Z_i\right)^2\right]=
    \sum_{i=1}^n \mathbb{E}[Z_i^2]+2 \sum_{i=1}^{n-1}\sum_{j=i+1}^n \mathbb{E}[Z_i Z_j]\leq
    nJ_1 +2 J_2
\end{equation}
where $J_1 = \mathbb{E}[Z_1^2]$ and $J_2 = \sum_{i=1}^{n-1}\sum_{r=1}^{n-i} |\mathbb{E}[Z_i Z_{i+r}]|$. Using (i) and (ii) of Lemma \ref{AAL1}, we have

\begin{eqnarray} \label{AAJ1}
      J_1 
    &\leq&
    \mathbb{E}\left[ (G(X_1 ,\varepsilon_1 ))^2 \right]\nonumber\\
    \nonumber
    &\leq&
    \frac{1}{n^2}\mathbb{E}\left[ (r(X_1)+\varepsilon_1)^2K_h^2 (x-X_1)g^{-2}(X_1) \right]\\
    \nonumber
    &\leq&
    \frac{1}{n^2}\mathbb{E}\left[ (r(X_1))^2 K_h^2 (x-X_1)g^{-2}(X_1) \right] + \frac{1}{n^2}\mathbb{E}\left[(\varepsilon_{1})^{2} K_h^2 (x-X_1)g^{-2}(X_1) \right]\\
     &\leq&
     \frac{1}{n^2}\frac{B_{2}}{h}+\frac{1}{n^2}\frac{B_{1}}{h}=\frac{A_{1}}{n^2 h}.
\end{eqnarray}
Let $(u_n)$ be a sequence such that $1\leq u_n \leq n-1$. We have
\begin{equation} \label{AACota1J2}
J_2 =\sum_{i=1}^{n-1}\left( \sum_{r=1}^{u_n} |Cov (Z_i, Z_{i+r})|+\sum_{r=u_n +1}^{n-i} |Cov (Z_i, Z_{i+r})|\right)
\end{equation}
Using inequality $(iii)$ from Lemma \ref{AAL3}, we obtain that
\begin{equation} \label{AATermino1J2}
   \sum_{r=1}^{u_n} |Cov (Z_i, Z_{i+r})|=\sum_{r=1}^{u_n} |Cov \left( G(X_i,\varepsilon_i ),G(X_{i+r},\varepsilon_{i+r}) \right)|
\leq \sum_{r=1}^{u_n} \frac{B_{5}}{n^2}
=
    B_{5} \frac{u_n}{n^2}
\end{equation}
Using Proposition \ref{AAP2}, we have that
\begin{equation*}
|Cov (Z_i, Z_{i+r})| 
\leq
\phi(1,1)D_1^0(h)D_2(h)\alpha_r^{\frac{1}{3}}
=
3\frac{B_{7}}{n^2 h}a^{\frac{r}{3}}.
\end{equation*}
Therefore, we have
\begin{equation} \label{AATermino2J2}
    \sum_{r=u_n +1}^{n-i} |Cov (Z_i, Z_{i+r})| 
    \leq
    \frac{3B_{7}}{n^2 h} \sum_{r=u_n +1}^{\infty}a^{\frac{r}{3}}
    =
    \frac{3B_{7}}{n^2 h}\frac{(a^\frac{1}{3})^{u_n +1}}{1-a^\frac{1}{3}}
\end{equation}
Due to inequalities \eqref{AACota1J2}, \eqref{AATermino1J2}, and \eqref{AATermino2J2}, we have

\begin{equation}\label{eq:J2}
J_2 \leq
\frac{1}{n}\left( B_{5} u_n +\frac{3B_{7}}{1-a^\frac{1}{3}}\frac{(a^\frac{1}{3})^{u_n +1}}{h} \right)
\end{equation}

In the proof of Proposition \ref{AApro:var}, the variable $u_n$ is chosen as $u_n = \left[\frac{1}{h|\log a|(\log n)^\frac{1}{2}}\right]$, and the facts that $[x] \leq x \leq [x]+1$ and $0<a<1$ are used to further bound the inequality.

\begin{eqnarray} \label{AACota2J2}
    \nonumber
    J_2 &\leq& \frac{1}{n}\left( B_{5} \frac{1}{h|\log a|(\log n)^\frac{1}{2}} +\frac{3B_{7}}{1-a^\frac{1}{3}}\frac{(a^\frac{1}{3})^{\left\{\frac{1}{h|\log a|(\log n)^\frac{1}{2}}\right\}}}{h} \right)\\
    &=&
    \frac{1}{nh}\left( \frac{B_{5}}{|\log a|(\log n)^\frac{1}{2}} +\frac{3B_{7}}{1-a^\frac{1}{3}}(a^\frac{1}{3})^{\left\{\frac{1}{h|\log a|(\log n)^\frac{1}{2}}\right\}} \right)
\end{eqnarray}

Since $h \in \mathcal{H}$, then $h \leq \frac{1}{(\log n)^2}$. Furthermore, due to the fact that $e^{-x}<\frac{1}{x}$ for every $x \in \mathbb{R}$ and $a\in]0,1[$, following a procedure analogous to that carried out in Lemma \ref{AAL2}, we arrive at,

$$ (a^\frac{1}{3})^{\left\{\frac{1}{h|\log a|(\log n)^\frac{1}{2}}\right\}} \leq \frac{3}{(\log n)^\frac{3}{2}}$$
for each $n\geq 2$.\\

Based on the previous result and inequality \eqref{AACota2J2}, we have,

\begin{eqnarray} \label{AAJ2}
    \nonumber
    J_2 &\leq& \frac{1}{nh}\left( \frac{B_{5}}{|\log a|(\log n)^\frac{1}{2}} +\frac{3B_{7}}{1-a^\frac{1}{3}}\frac{3}{(\log n)^\frac{3}{2}} \right)\\
    &\leq&
   \left( \frac{B_{5}}{|\log a|} +\frac{9B_{7}}{1-a^\frac{1}{3}} \right) \frac{(\log n)^{-\frac{1}{2}}}{nh}
\end{eqnarray}

Now, using inequalities \eqref{AAVarSnDB0}, \eqref{eq:J2}, and \eqref{AAJ2}, we obtain the following bound on the variance of $S_n$,
\begin{equation} \label{AAVarSnDB}
    Var[S_n] \leq n\frac{A_{1}}{n^2 h} +2 \left( \frac{B_{5}}{|\log a|} +\frac{9B_{7}}{1-a^\frac{1}{3}} \right) \frac{(\log n)^{-\frac{1}{2}}}{nh}=
    \mathcal{A}_n
\end{equation}
where $\mathcal{A}_n=\frac{A_{1}}{nh} + B_{8} \frac{(\log n)^{-\frac{1}{2}}}{nh}$.\\

Finally, since $\mathcal{A}_n \geq \frac{A_{1}}{nh}$ and considering the expression for $\mathcal{B}^*_n$ given in \eqref{eq:Bn}, we have that

$$\mathcal{B}^*_{n} \leq 2(\sqrt{D_2} \vee D_1){L}_2\left(\frac{2^5 nD_2{L}_1}{\frac{A_{1}}{nh}}\vee 1\right). $$
Using the fact that for large $n$, $M_{n}\geq 1$, we have

$$\sqrt{D_2} \vee D_1\le (\sqrt{B_7} \vee B_6)\frac{M_n}{nh} $$
This allows us to obtain that

$$\mathcal{B}^*_{n} \leq \mathcal{B}_n.$$

\subsubsection{Proof of Corollary \ref{AACDBernstein}.}

Since $\lambda(u)=(2\mathcal{A}_n u)^\frac{1}{2}+\mathcal{B}_n (2u)^3$, then
\begin{equation} \label{AAlam0}
    \lambda^2 (u) =\left[(2\mathcal{A}_n u)^\frac{1}{2}+\mathcal{B}_n (2u)^3\right]^2 \geq 2u \mathcal{A}_n
\end{equation}
analogously

\begin{equation*}
    \lambda^\frac{1}{3} (u) \geq 2u \mathcal{B}_n^\frac{1}{3}
\end{equation*}

Multiplying the previous inequality by $\lambda^\frac{5}{3} (u)$, we obtain,
\begin{equation*}
    \lambda^2 (u) \geq 2u \mathcal{B}_n^\frac{1}{3} \lambda^\frac{5}{3} (u)
\end{equation*}

Now, by adding inequality \eqref{AAlam0} to the previous inequality, we obtain,
\begin{equation*}
    \lambda^2 (u) \geq u \left[ \mathcal{A}_n + \mathcal{B}_n^\frac{1}{3} \lambda^\frac{5}{3} (u) \right]
\end{equation*}

By the previous inequality and taking the variable change $t=\lambda(u)$, it follows that,
\begin{equation*}
    \frac{\frac{t^2}{2}}{\mathcal{A}_n +\mathcal{B}_n^\frac{1}{3}t^\frac{5}{3}} = \frac{1}{2} \frac{\lambda^2 (u)}{\mathcal{A}_n +\mathcal{B}_n^\frac{1}{3}\lambda^\frac{5}{3} (u)} \geq \frac{u}{2}
\end{equation*}

Considering the Bernstein inequality stated in Proposition \ref{AATDBernstein}, the aforementioned inequality and the change of variables $t=\lambda (u)$ imply that,

\begin{equation*}
    \mathbb{P}\left(|S_n | \geq \lambda(u)\right) \leq \exp \left(-\frac{u}{2} \right) 
\end{equation*}
for each $u \geq 0$.

\section{Bias and Variance of the estimator.}\label{AAControl:sesgoYvar}
\vspace{0.5cm}
\subsection{Proof of Proposition \ref{AApro:sesgo}: Bias of the estimator \texorpdfstring{$\hat{r}_h$}{rh}.}
\label{AAdemopro:sesgo}

\begin{eqnarray*}
    \mathbb{E}\{\hat{r}_h (x)\} &=& \mathbb{E}\left\{\frac{1}{n}\sum_{i=1}^n K_h(x-X_i )Y_i  g^{-1}(X_i)\right\}\\
    &=&
    \mathbb{E}\left\{\frac{1}{n}\sum_{i=1}^n K_h(x-X_i )r(X_i)g^{-1}(X_i)\right\} + \mathbb{E}\left\{\frac{1}{n}\sum_{i=1}^n K_h(x-X_i )\varepsilon_i  g^{-1}(X_i)\right\}\\
    &=&
    \mathbb{E}\left\{K_h(x-X)r(X)g^{-1}(X)\right\} + \mathbb{E}\left\{K_h(x-X)g^{-1}(X)\right\} \mathbb{E}\{\varepsilon\}\\
    &=&
    \int_{\mathbb{R}}K_h (x-t)r(t)dt=K_h \ast r(x)
\end{eqnarray*}
Thus, it is obtained that,
\begin{equation*}
 \mathbb{E}[\hat{r}_h (x)]-r(x)=
   K_h \ast r(x)-r(x)=
   \int_{\mathbb{R}}\left(r(x-hu)-r(x)\right)K\left(u\right)du
\end{equation*}
From the Taylor-Lagrange formula up to order $\mathit{l}$ applied to $r(x)$, it follows that,
\begin{equation*}
r(x-hu)=r(x)-hur'(x)+\frac{(-hu)^{2}}{2!}r''(x)-\cdots+\frac{(-hu)^{\mathit{l}-1}}{(\mathit{l}-1)!}r^{(\mathit{l}-1)}(x)+\frac{(-hu)^{\mathit{l}}}{\mathit{l}!}r^{(\mathit{l})}(x-\zeta_u hu)
\end{equation*}
where $\zeta_u \in [0,1]$.\\

Due to the above and since $K$ is a kernel of order $m$ with $m \geq \mathit{l}$, it follows that,
\begin{equation*}
\mathbb{E}[\hat{r}_h (x)]-r(x)=\int_{\mathbb{R}}\frac{(-hu)^{\mathit{l}}}{\mathit{l}!}r^{(\mathit{l})}(x-\zeta_u hu)K(u)du
\end{equation*}

Given that $K$ is a kernel of order $\mathit{l}$, by subtracting the null term $\int_{\mathbb{R}}\frac{(-uh)^{\mathit{l}}}{\mathit{l}!}r^{(\mathit{l})}(x)K(u)du$ from the previous equation, it follows that,
\begin{equation*}
\mathbb{E}[\hat{r}_h (x)]-r(x)=(-1)^l\frac{h^{\mathit{l}}}{\mathit{l}!}\int_{\mathbb{R}}\left(r^{(\mathit{l})}(x-\zeta_u hu)-r^{(\mathit{l})}(x)\right)u^{\mathit{l}}K(u)du
\end{equation*}

Taking the absolute value on both sides and using the hypothesis $r \in \Sigma(\beta, L)$, we have,

\begin{align*}
\left|\mathbb{E}[\hat{r}_h (x)]-r(x)\right| &\leq\frac{h^{\beta}}{\mathit{l}!}L\int_{\mathbb{R}}|u|^{\beta}\zeta_u^{\beta-\mathit{l}}|K(u)|du\\ &\leq\frac{h^{\beta}}{\mathit{l}!}L\int_{\mathbb{R}}|u|^{\beta}|K(u)|du, \quad  \text{pues} \quad   \zeta_u^{\beta-\mathit{l}}<1\\
&\leq h^{\beta}A_{0}.
\end{align*}

\subsection{Proof of Proposition \ref{AApro:var}: Variance of the estimator \texorpdfstring{$\hat{r}_h$}{rh}.} \label{AAdemopro:var}

By substituting $Y_i$ with $r(X_i)+\varepsilon_i$, we obtain,

\begin{eqnarray} \label{AAExpreSigma^2}
\nonumber
\lefteqn{\mathbb{E}\left[(\hat{r}_h (x)-\mathbb{E}[\hat{r}_h (x))^2\right]}\\
    &=&
    \nonumber
    \mathbb{E}\left[\left\{\frac{1}{n}\sum_{i=1}^n K_h(x-X_i )Y_i  g^{-1}(X_i) - \mathbb{E}\left(\frac{1}{n}\sum_{i=1}^n K_h(x-X_i )Y_i  g^{-1}(X_i)\right)\right\}^2\right]\\
    \nonumber
    &=&
    \mathbb{E}\left\{\frac{1}{n}\sum_{i=1}^n \left[ K_h(x-X_i )r(X_i )g^{-1}(X_i) - \mathbb{E}\left(K_h(x-X_i )r(X_i )g^{-1}(X_i)\right)\right] \right.\\
    \nonumber
    &+&
    \left. \frac{1}{n}\sum_{i=1}^n K_h(x-X_i )\varepsilon_i g^{-1}(X_i)\right\}^2\\
    &=&
    I_1 +I_2 +I_3
\end{eqnarray}
where,

\begin{itemize}
     \item $A(X_1, \cdots ,X_n )=\sum_{i=1}^n \left[ K_h(x-X_i )r(X_i )g^{-1}(X_i) - \mathbb{E}(K_h(x-X_i )r(X_i )g^{-1}(X_i))\right]$ 
      \item $I_1 =  \mathbb{E}\left\{\frac{1}{n}\sum_{i=1}^n \left[ K_h(x-X_i )r(X_i )g^{-1}(X_i) - \mathbb{E}\left(K_h(x-X_i )r(X_i )g^{-1}(X_i)\right)\right]\right\}^2$
      \item $I_2 = \mathbb{E}\left\{2 \frac{1}{n}A(X_1, \cdots ,X_n ) \frac{1}{n}\sum_{i=1}^n K_h(x-X_i )\varepsilon_i g^{-1}(X_i) \right\}$
      \item $I_3 =\mathbb{E}\left\{\frac{1}{n}\sum_{i=1}^n K_h(x-X_i )\varepsilon_i g^{-1}(X_i) \right\}^2$
\end{itemize}

We proceed to bound $I_1$, $I_2$, and $I_3$.

\begin{eqnarray} \label{AACotaI3}
    \nonumber
    I_3 
    &=& \frac{1}{n^2}\mathbb{E}\left\{\sum_{i=1}^n K_h(x-X_i )\varepsilon_i g^{-1}(X_i) \right\}^2\\
    \nonumber
    &=&
    \frac{1}{n^2}\sum_{i=1}^n \mathbb{E}\left\{K_h(x-X_i )\varepsilon_i g^{-1}(X_i) \right\}^2 \\
    \nonumber
    && +  \frac{1}{n^2}\sum_{i\neq j}^n \mathbb{E}\left\{K_h(x-X_i )\varepsilon_i g^{-1}(X_i) K_h(x-X_j )\varepsilon_j g^{-1}(X_j)\right\}\\
    \nonumber
    &=&
    \frac{1}{n^2}\sum_{i=1}^n \frac{B_{1}}{h} +  \frac{1}{n^2}\sum_{i\neq j}^n \mathbb{E}\left\{K_h(x-X_i ) g^{-1}(X_i) K_h(x-X_j ) g^{-1}(X_j)\right\}\mathbb{E}( \varepsilon_i\varepsilon_j)\\
    &=&
    \frac{B_{1}}{nh}
\end{eqnarray}
by the first inequality of Lemma \ref{AAL1}. 

\begin{eqnarray} \label{AACotaI2}
   \nonumber
    I_2 &=& \frac{2}{n^2} \mathbb{E}\left\{ A(X_1, \cdots ,X_n )\sum_{i=1}^n K_h(x-X_i )\varepsilon_i g^{-1}(X_i) \right\}\\
    &=&
    \frac{2}{n^2}  \sum_{i=1}^n \mathbb{E}\left\{A(X_1, \cdots ,X_n )K_h(x-X_i ) g^{-1}(X_i) \right\}\mathbb{E}\{\varepsilon_i\}=0
\end{eqnarray}

\begin{eqnarray} \label{AAExpreI1}
   \nonumber
   I_1 &=&  \frac{1}{n^2}\mathbb{E}\left\{\sum_{i=1}^n \left[ K_h(x-X_i )r(X_i )g^{-1}(X_i) - \mathbb{E}(K_h(x-X_i )r(X_i )g^{-1}(X_i))\right]\right\}^2\\
   &\leq&
   I_4 + I_5
\end{eqnarray}

where,

\begin{itemize}
     \item $I_4 =\frac{1}{n^2}\sum_{i=1}^n Var \left[ K_h(x-X_i )r(X_i )g^{-1}(X_i)\right]$ 
      \item $I_5 = \frac{1}{n^2}\sum_{i\neq j}^n \left| Cov \left[K_h(x-X_i )r(X_i )g^{-1}(X_i), K_h(x-X_j )r(X_j )g^{-1}(X_j)\right] \right|$
\end{itemize}

We proceed to bound $I_4$ and $I_5$. By the second inequality of Lemma \ref{AAL1}.

\begin{equation} \label{AACotaI4}
   I_4 \leq  \frac{1}{n^2}\sum_{i=1}^n \mathbb{E} \left[ \left( K_h(x-X_i )r(X_i )g^{-1}(X_i) \right)^2\right]
   \leq
   \frac{B_{2}}{nh} 
\end{equation}

\begin{eqnarray} \label{AAExpreI5}
   \nonumber
   I_5 
    &=&
   \frac{2}{n^2}\sum_{i=1}^{n-1}\sum_{k=1}^{n-i} |Cov (\tilde{X}_i, \tilde{X}_{i+k})|\\
    &=&
   \frac{2}{n^2}\sum_{i=1}^{n-1}\left( \sum_{k=1}^{u_n} |Cov (\tilde{X}_i, \tilde{X}_{i+k})|+\sum_{k=u_n +1}^{n-i} |Cov (\tilde{X}_i, \tilde{X}_{i+k})|\right)
\end{eqnarray}
where $1\leq u_n \leq n-1$ and $\tilde{X}_i =K_h(x-X_i )r(X_i )g^{-1}(X_i)$, $\forall i$.\\

We have, by the third and fourth inequalities of Lemma \ref{AAL1}.
\begin{eqnarray} \label{AACotaI_i}
   \nonumber
   \lefteqn{\sum_{k=1}^{u_n} |Cov (\tilde{X}_i, \tilde{X}_{i+k})| }\\ 
   \nonumber
   &\leq& \sum_{k=1}^{u_n} |\mathbb{E}[\tilde{X}_i \tilde{X}_{i+k}]|+|\mathbb{E}[\tilde{X}_i]\mathbb{E}[\tilde{X}_{i+k}]|\\
   \nonumber
   &\leq&
    \sum_{k=1}^{u_n} \left\{ \mathbb{E}[|K_h(x-X_i )r(X_i )g^{-1}(X_i) K_h(x-X_{i+k} )r(X_{i+k} )g^{-1}(X_{i+k})|] \right.\\
    \nonumber
    &+&
    \left. \left( \mathbb{E}[|\tilde{X}_i|] \right)^2 \right\}\\
    \nonumber
     &\leq&
    \sum_{k=1}^{u_n} \left\{ (r_{\sup})^2(g_{\inf})^{-2}\mathbb{E}[|K_h(x-X_i ) K_h(x-X_{i+k})|]+\left( r_{\sup} \|K\|_1 \right)^2 \right\}\\
     &=&
    \sum_{k=1}^{u_n} \left\{ (r_{\sup})^2(g_{\inf})^{-2}Q\|K\|_1^2 + (r_{sup})^2 \|K\|_1^2 \right\}= A_3 u_n,
\end{eqnarray}

To bound the term $\sum_{k=u_n +1}^{n-i} |Cov (\tilde{X}_i, \tilde{X}_{i+k})|$ in equality \eqref{AAExpreI5}, we use the fact that the $X_{i}$ satisfy $(H_3)$.\\

As,
\begin{eqnarray*}
   |\tilde{X}_i |&=& |K_h (x-X_i )r(X_i)g^{-1}(X_i )|\\
   &\leq&
   r_{\sup}(g_{\inf})^{-1}\left|\frac{1}{h}K\left(\frac{x-X_i}{h}\right)\right|\\
   &\leq&
   \frac{r_{\sup}(g_{\inf})^{-1}\|K\|_{\infty}}{h}
\end{eqnarray*}

By the previous inequality, equation \eqref{AACoefDevil}, and hypothesis $(H_{3})$,
\begin{eqnarray*}
   |cov(\tilde{X}_i ,\tilde{X}_{i+k})| &\leq& 4 a^k \frac{(r_{\sup})^2(g_{\inf})^{-2}\|K\|_{\infty}^2}{h^2}
\end{eqnarray*}

By this last inequality, one has,
\begin{eqnarray} \label{AACotaJ_i}
   \nonumber
   \sum_{k=u_n +1}^{n-i} |Cov (\tilde{X}_i, \tilde{X}_{i+k})| &\leq& \frac{4(r_{\sup})^2(g_{\inf})^{-2}\|K\|_{\infty}^2}{h^2}\sum_{k=u_n +1}^{\infty} a^k\\
   \nonumber
   &=&
   \frac{4(r_{\sup})^2(g_{\inf})^{-2}\|K\|_{\infty}^2}{h^2}\frac{a^{u_n +1}}{1-a}\\
   &=&
   A_4\frac{a^{u_n +1}}{h^2}.
\end{eqnarray}

By equation \eqref{AAExpreI5} and inequalities \eqref{AACotaI_i} and \eqref{AACotaJ_i}, one has,
\begin{eqnarray*}
   I_5  &\leq&  \frac{2}{n^2}\sum_{i=1}^{n-1}\left( A_3 u_n + A_4\frac{a^{u_n +1}}{h^2}\right)\\
   &\leq&
   \frac{1}{n}\left(2A_3 u_n + 2A_4\frac{a^{u_n +1}}{h^2}\right)
\end{eqnarray*}

The variable $u_n$ in the summations of equation \eqref{AAExpreI5} is taken as $u_n = \left[\frac{1}{h|\log a|(\log n)^\frac{1}{2}}\right]$, and the facts that $[x] \leq x \leq [x]+1$ and $0<a<1$ are used to continue bounding the preceding inequality

\begin{eqnarray}
   \nonumber
   I_5  &\leq& 
   \frac{1}{n}\left(2A_3 \frac{1}{h|\log a|(\log n)^\frac{1}{2}} + 2A_4\frac{a^{\frac{1}{h|\log a|(\log n)^\frac{1}{2}}}}{h^2}\right)\\
   \nonumber
   &=&
   \frac{1}{nh}\left(2A_3 \frac{1}{|\log a|(\log n)^\frac{1}{2}} + 2A_4\frac{1}{h}a^{\frac{1}{h|\log a|(\log n)^\frac{1}{2}}}\right)
\end{eqnarray}

By Lemma \ref{AAL2} and the preceding inequality, it follows that, for each $n\geq 4$,

\begin{eqnarray} \label{AACotaI5}
    \nonumber
   I_5  &\leq& \frac{1}{nh}\left(\frac{2A_3}{|\log a|} \frac{1}{(\log n)^\frac{1}{2}} + 2A_4 \frac{10}{(\log n)^\frac{5}{2}}\right)\\
    &\leq&
    \left(\frac{2A_3}{|\log a|} + 20A_4\right)\frac{(\log n)^{-\frac{1}{2}}}{nh}.
\end{eqnarray}

By equalities \eqref{AAExpreI1}, \eqref{AACotaI4}, and \eqref{AACotaI5}, for all $n\geq 4$,
\begin{equation} \label{AACotaI1}
   I_1  \leq 
  \frac{B_{2}}{nh}+ \left(\frac{2A_3}{|\log a|} + 20A_4\right)\frac{(\log n)^{-\frac{1}{2}}}{nh}.
\end{equation}

Finally, by equation \eqref{AAExpreSigma^2} and inequalities \eqref{AACotaI3}, \eqref{AACotaI2}, and \eqref{AACotaI1}, it follows that, for each $n\geq 4$,
\begin{equation} \label{AAvar:hatr}
   \nonumber
  \mathbb{E}\left[(\hat{r}_h (x)-\mathbb{E}[\hat{r}_h (x))^2\right] = I_1 +I_3\leq 
    \frac{A_{1}}{nh}+A_2\frac{(\log n)^{-\frac{1}{2}}}{nh}.
\end{equation}

\section{Oracle Inequality Propositions.}\label{AAapp:desi}
\vspace{0.5cm}
\subsection{Proof of Proposition \ref{AAPropDO}.}\label{AAdemopropDO}

Let $h\in\mathcal{H}$. We have
\begin{equation*}
|r(x)-\hat{r}_{\hat{h}}(x)|
\leq|r(x)-\hat{r}_{h}(x)|+|\hat{r}_{h}(x)-\hat{r}_{h,\hat{h}}(x)|\\
+|\hat{r}_{\hat{h}}(x)-\hat{r}_{h,\hat{h}}(x)|.
\end{equation*}
By adding and subtracting $V(h)$ and $V(\hat{h})$, we obtain
\begin{eqnarray}
\label{AAPreCoraRies}
\nonumber
|r(x)-\hat{r}_{\hat{h}}(x)|
&\leq&|r(x)-\hat{r}_{h}(x)|+A(h,x)+V(h)+A(\hat{h},x)+V(\hat{h})\\
&\leq&|r(x)-\hat{r}_{h}(x)|+2(A(h,x)+V(h)).
\end{eqnarray}

On the other hand,
\begin{eqnarray*}
|\hat{r}_{h,h'}(x)-\hat{r}_{h'}(x)|
&\leq&|\hat{r}_{h,h'}(x)-\mathbb{E}(\hat{r}_{h,h'}(x))|+|\hat{r}_{h'}(x)
-\mathbb{E}(\hat{r}_{h'}(x))|\\
&+&|\mathbb{E}(\hat{r}_{h,h'}(x))-\mathbb{E}(\hat{r}_{h'}(x))|
\end{eqnarray*}
and if we subtract $V(h')$ from both sides and take the positive part, we have,
\begin{eqnarray}
\label{AACotaVarG-L}
\nonumber
\lefteqn{\left\{|\hat{r}_{h,h'}(x)-\hat{r}_{h'}(x)|-V(h')\right\}_+}\\
\nonumber
&\leq&
\left\{|\hat{r}_{h,h'}(x)-\mathbb{E}(\hat{r}_{h,h'}(x))|+|\hat{r}_{h'}(x)
-\mathbb{E}(\hat{r}_{h'}(x))|-V(h')\right\}_+\\
&+&
|\mathbb{E}(\hat{r}_{h,h'}(x))-\mathbb{E}(\hat{r}_{h'}(x))|
\end{eqnarray}
Remember that the estimator's expectation is $\mathbb{E}(\hat{r}_h(x))=(K_h\ast r)(x)$,
so the second term on the right side of the previous inequality is expressed as
\begin{eqnarray*}
|\mathbb{E}(\hat{r}_{h,h'}(x))-\mathbb{E}(\hat{r}_{h'}(x))|&=&|(K_{h'}\ast r)(x)-(K_{h'}\ast(K_{h}\ast r))(x)|\\
&=&|(K_{h'}\ast(r-(K_{h}\ast r))(x)|\\
&\leq&\int |K_{h'}(x-u)(K_h\ast r(u)-r(u))|du .  
\end{eqnarray*}

Since the support of $K$ is $[-1,1]$, we have $\frac{x-u}{h'}\in[-1,1]$, implying that $u\in B(x)$, which implies
\begin{eqnarray*}
 |\mathbb{E}(\hat{r}_{h,h'}(x))-\mathbb{E}(\hat{r}_{h'}(x))| &\leq& \int_{B(x)} \left| \frac{1}{h'}K\left(\frac{x-u}{h'}\right) \right|\left|(K_h\ast r(u)-r(u))\right|du\\ &\leq& \max_{u\in B(x)}\left|(K_h\ast r(u)-r(u))\right|\int_{} \left| \frac{1}{h'}K\left(\frac{x-u}{h'}\right)\right| du\\ &\leq& \|K\|_1 \max_{u\in B(x)}\left|(K_h\ast r(u)-r(u))\right|\\
& \leq& C(h)\|K\|_1.
\end{eqnarray*}

Using that $V(h')=V_1(h')+V_2(h')$ and the previous inequality, we have,
\begin{eqnarray}
\label{AAPreCotaA(h)}
\nonumber
\lefteqn{\left\{|\hat{r}_{h,h'}(x)-\hat{r}_{h'}(x)|-V(h')\right\}_+}\\
\nonumber
&\leq&
\left\{|\hat{r}_{h,h'}(x)-\mathbb{E}(\hat{r}_{h,h'}(x))|-V_{2}(h')\right\}_+ +\left\{|\hat{r}_{h'}(x)
-\mathbb{E}(\hat{r}_{h'}(x))|-V_{1}(h')\right\}_+\\ &+&
C(h)\|K\|_1.
\end{eqnarray}

By the definition of $A(h,x)$ and by taking the maximum over $h'\in\mathcal{H}$ in inequality \eqref{AAPreCotaA(h)}, we have,
\begin{equation}\label{AACotaA(h,x)}
    A(h,x)\leq T_{1}+T_{2}+C(h)\|K\|_1.
\end{equation}

When taking the norm $\E[\cdot^{2}]^{\frac{1}{2}}$ determined by the expectation in equations \eqref{AAPreCoraRies} and \eqref{AACotaA(h,x)}, we obtain
\begin{eqnarray*}\label{AACotaRiesPunt}
\lefteqn{\left(\mathbb{E}\left(r(x)-\hat{r}_{\hat{h}}(x)\right)^2\right)^{\frac{1}{2}}}\\
\nonumber
&\leq& \left(\mathbb{E}\left(r(x)-\hat{r}_{h}(x)\right)^2\right)^\frac{1}{2}+2\left(\left(\mathbb{E}(A(h,x))^2\right)^\frac{1}{2}+V(h)\right)\\
&\leq&  \left(\mathbb{E}\left(r(x)-\hat{r}_{h}(x)\right)^2\right)^\frac{1}{2}+C(h)\|K\|_1+V(h)+\left(\mathbb{E}(T_{1}^{2})\right)^{\frac{1}{2}}+\left(\mathbb{E}(T_{2}^{2})\right)^{\frac{1}{2}}
\end{eqnarray*}

\subsection{Proof of Proposition \ref{AAPropET_1^2yV_1}, (i).}\label{AAdemoT1}

Let's prove that for $n$ sufficiently large,
 $$\mathbb{E}(T_1^2) \leq A_5\frac{\log n}{n}$$
 where $A_5 = B_{9} +2B_{10}$, with $B_{9} =(g_{\inf})^{-2}\|K\|_{\infty}^2 \left\{(8((r_{\sup})^4 +3\sigma^4)\right\}^{\frac{1}{2}}\left(\frac{2}{\sqrt{2\pi}}\right)^{\frac{1}{2}}$, and 
 $B_{10} = \frac{2^4 3}{1-e^{-\frac{\delta }{2}}}(A_{1} + C_{6}) + \frac{5!2^{18} 3 B^2}{1-e^{-\frac{\delta }{2}}} (\sigma + r_{\sup})^2$.
 
To study $\mathbb{E}(T_1^2)$, the following fact is used,
\begin{eqnarray}
\label{AATemporal}
\nonumber
\mathbb{E}(T_1^2) &\leq& \max_{h'\in\mathcal{H}}\mathbb{E}\left[\left\{|\hat{r}_{h'}(x)
-\mathbb{E}(\hat{r}_{h'}(x))|-V_{1}(h')\right\}_+^2\right]\\ &\leq& \sum_{h'\in\mathcal{H}}\mathbb{E}\left[\left\{|\hat{r}_{h'}(x)
-\mathbb{E}(\hat{r}_{h'}(x))|-V_{1}(h')\right\}_+^2\right]
\end{eqnarray}

In this case, there is no guarantee that the term inside the summation on the right side of inequality \eqref{AATemporal} is bounded, preventing the application of the Bernstein inequality to bound it. To address this issue, the truncation auxiliary estimator is defined as follows.
$$
\Tilde{r}_h(x)=\frac{1}{n}\sum_{k=1}^nY_kK_h(x-X_k)g^{-1}(X_k)\mathbf{1}_{\{|Y_k |\leq M_n\}}
$$
for $h\in\mathcal{H}$ with $M_n >0$, $\forall n \in \mathbb{N}$.\\

The term $\left\{|\hat{r}_{h'}(x) -\mathbb{E}(\hat{r}_{h'}(x))|-V_{1}(h')\right\}_+^2$ is decomposed into two parts using the truncation auxiliary estimator and its expectation, as shown below.
$$\hat{r}_{h'}(x)
-\mathbb{E}(\hat{r}_{h'}(x))=\Tilde{r}_{h'}(x)
-\mathbb{E}(\Tilde{r}_{h'}(x))+\{\hat{r}_{h'}(x)-\Tilde{r}_{h'} (x)-\mathbb{E}\left(\hat{r}_{h'}(x)-\Tilde{r}_{h'}(x)\right)\}$$
thus, we have that,
$$|\hat{r}_{h'}(x)
-\mathbb{E}(\hat{r}_{h'}(x))|\leq |\Tilde{r}_{h'}(x)
-\mathbb{E}(\Tilde{r}_{h'}(x))|+|\hat{r}_{h'}(x)-\Tilde{r}_{h'} (x)-\mathbb{E}\left(\hat{r}_{h'}(x)-\Tilde{r}_{h'}(x)\right)|$$
therefore,

\begin{eqnarray*}
\{|\hat{r}_{h'}(x)
-\mathbb{E}(\hat{r}_{h'}(x))|-V_1 (h')\}_+^2 &\leq& 2\{|\Tilde{r}_{h'}(x)
-\mathbb{E}(\Tilde{r}_{h'}(x))|-V_1 (h')\}_+^2\\ &+& 
2\left\{\hat{r}_{h'}(x)-\Tilde{r}_{h'} (x)-\mathbb{E}\left(\hat{r}_{h'}(x)-\Tilde{r}_{h'}(x)\right)\right\}^2
\end{eqnarray*}
upon taking the expectation, we get,

\begin{eqnarray*}
\mathbb{E}\left( \{|\hat{r}_{h'}(x)
-\mathbb{E}(\hat{r}_{h'}(x))|-V_1 (h')\}_+^2\right) \leq 2\Delta_1(h') + 2\Delta_2(h')
\end{eqnarray*}
where 
$$\Delta_1(h')=\mathbb{E}\left( \{|\Tilde{r}_{h'}(x)
-\mathbb{E}(\Tilde{r}_{h'}(x))|-V_1 (h')\}_+^2\right), $$
$$
\Delta_2(h')=\mathbb{E}\left( \left\{\hat{r}_{h'}(x)-\Tilde{r}_{h'} (x)-\mathbb{E}\left(\hat{r}_{h'}(x)-\Tilde{r}_{h'}(x)\right)\right\}^2\right)$$

Upon substituting the previous result into inequality \eqref{AATemporal}, we have that,

\begin{eqnarray}
\label{AATemporal1}
\mathbb{E}(T_1^2) &\leq& 2\left( \sum_{h'\in\mathcal{H}}\Delta_1(h') + \sum_{h'\in\mathcal{H}}\Delta_2(h') \right)
\end{eqnarray}

We proceed to bound the summation $\sum_{h'\in\mathcal{H}}\Delta_2(h')$. For this purpose, we analyze the term $\hat{r}_{h'}(x)-\Tilde{r}_{h'}(x)$,

\begin{eqnarray*}
\hat{r}_{h'}(x)-\Tilde{r}_{h'}(x) &=& \frac{1}{n}\sum_{k=1}^nY_kK_{h'}(x-X_k)g^{-1}(X_k) - \frac{1}{n}\sum_{k=1}^nY_kK_{h'}(x-X_k)g^{-1}(X_k)\mathbf{1}_{\{|Y_k |\leq M_n\}}\\ 
&=&
\frac{1}{n}\sum_{k=1}^n\eta_k
\end{eqnarray*}
where $\eta_k = Y_kK_{h'}(x-X_k)g^{-1}(X_k)\mathbf{1}_{\{|Y_k |> M_n\}}$.\\

So, we have

\begin{eqnarray*}
\Delta_2(h') &=& \frac{1}{n^2} var\left( \sum_{k=1}^n\eta_k\right) \le \frac{1}{n^2} \mathbb{E}\left(\left\{ \sum_{k=1}^n\eta_k\right\}^2\right)   \le \frac{1}{n^2} \mathbb{E}\left( n\sum_{k=1}^n\eta_k^2\right) \le 
\mathbb{E}\left(\eta_1^2\right)
\end{eqnarray*}


and

\begin{eqnarray*} \label{AAB0}
\mathbb{E}\left(\eta_1^2\right)&= & \frac{1}{(h')^2}\mathbb{E}\left\{  Y^2 K^2 \left(\frac{x-X}{h'} \right)g^{-2}(X)\mathbf{1}_{\{|Y |> M_n\}}\right\}\\
&\leq &\frac{(g_{\inf})^{-2} \|K\|_{\infty}^2}{(h')^2} \mathbb{E}\left\{ (Y^2)^2 \right\}^{\frac{1}{2}} \mathbb{E}\left\{ \mathbf{1}_{\{|Y |> M_n\}}^2\right\}^{\frac{1}{2}}\\
&\leq&
\frac{(g_{\inf})^{-2}\|K\|_{\infty}^2}{(h')^2} \mathbb{E}\left\{ Y^4 \right\}^{\frac{1}{2}} \mathbb{P}\left\{ \{|Y |> M_n\} \right\}^{\frac{1}{2}},
\end{eqnarray*}
with

\begin{equation*}
\label{AAm4Y}
\mathbb{E}\{Y^4\} 
\leq
8\mathbb{E}\left\{ r^4 (X)+ \varepsilon^4\right\}
\leq
8\left( (r_{sup})^4 + \mathbb{E}\{\varepsilon^4\}\right)\leq
8\left( (r_{sup})^4 + 3\sigma^4 \right)
\end{equation*}
and
\begin{eqnarray*}
\mathbb{P}\left( \{|Y |> M_n\}\right) &=& \mathbb{P}\left( \{|r(X)+\varepsilon |> M_n\}\right)\\
&\leq&
\mathbb{P}\left( \{r_{\sup}+|\varepsilon |> M_n\}\right)\\
&\leq&
\mathbb{P}\left( \left\{ |\varepsilon |> M_n - r_{\sup}\right\} \right)\\
&\leq&
2\mathbb{P}\left( \left\{\frac{\varepsilon}{\sigma}  > \frac{ M_n - r_{\sup}}{\sigma}\right\}\right)\\
&\leq&
2\frac{\phi\left( \log n\right)}{\log n}
\end{eqnarray*}
where $\phi(x)=\frac{1}{\sqrt{2\pi}}e^{-\frac{x^2}{2}}$.\\

This implies that
\begin{equation*}
\Delta_2(h')\le \frac{B_{9}}{(h')^2}\frac{e^{-\frac{(\log n)^2}{4}}}{(\log n) ^{\frac{1}{2}}}
\end{equation*}
and therefore,

\begin{equation*}
    \sum_{h'\in\mathcal{H}}\Delta_2(h') \leq  B_{9} \frac{e^{-\frac{(\log n)^2}{4}}}{(\log n) ^{\frac{1}{2}}}  \sum_{h'\in \mathcal{H}}\frac{1} {(h')^2}.
\end{equation*}

Since the family of bandwidths is $\mathcal{H}=\{e^{-i}\}_{i=0}^{M}\cap [h_{\min},h_{\max}]$ with $h_{\min} = \frac{(\log n)^8}{n}$, $h_{\max} = \frac{1}{(\log n)^2}$, and $M=\left[\log\left(\frac{n}{(\log n)^8}\right)\right]$, we have

\begin{equation*}
\sum_{h'\in \mathcal{H}}\frac{1} {(h')^2} \leq
\sum_{i=0}^{M}e^{2i}\leq
\int_0^{M}e^{2x}dx
\leq 
 \frac{e^{2M}}{2} 
\leq
\frac{n^2}{2}.
\end{equation*}
Using
$$ \frac{e^{-\frac{(\log n)^2}{4}}}{(\log n)^{\frac{1}{2}}}n^2 = \frac{e^{-\frac{(\log n)^2}{4}}}{(\log n)^{\frac{1}{2}}}e^{2\log n} =  \frac{e^{-\log n\left(\frac{\log n}{4}-2\right)}}{(\log n)^{\frac{1}{2}}} \leq \frac{e^{-\log n}}{(\log n)^{\frac{1}{2}}} \leq \frac{1}{n}$$
we finally obtain that

\begin{equation} \label{AACotaSumBh'}
    \sum_{h'\in\mathcal{H}}\Delta_2(h')  \leq \frac{B_{9}}{2n}.
\end{equation}

Now we proceed to bound the summation $\sum_{h'\in\mathcal{H}}\Delta_1(h')$ in inequality \eqref{AATemporal1}, where the general term of the summation $\Delta_1(h')=\mathbb{E}\left( \{|\Tilde{r}_{h'}(x) -\mathbb{E}(\Tilde{r}_{h'}(x))|-V_1 (h')\}_+^2\right)$ is expressed in terms of the truncation auxiliary estimator, ensuring that such a term is bounded.\\

We analyze the term $\Delta_1(h')$ as shown below.

\begin{eqnarray}
\label{AARiesgoControl}
\nonumber
\Delta_1(h') &=& \mathbb{E}\left(\left\{|\Tilde{r}_{h'}(x)-\mathbb{E}(\tilde{r}_{h'}(x))|-V_1 (h')\right\}_+^2\right)\\
\nonumber
&=&
\int_0^{+\infty}\mathbb{P}\left(\left\{|\tilde{r}_{h'}(x)-\mathbb{E}(\tilde{r}_{h'}(x))|-V_1 (h')\right\}_+^2>t\right)dt\\
\nonumber
&=& \int_0^{+\infty}\mathbb{P}\left(\left\{|\tilde{r}_{h'}(x)-\mathbb{E}(\tilde{r}_{h'}(x))|-V_1 (h')\right\}_{+}>t^{\frac{1}{2}}\right)dt\\
&=&\int_0^{+\infty}\mathbb{P}\left(|\tilde{r}_{h'}(x)-\mathbb{E}(\tilde{r}_{h'}(x))|>V_1(h')+t^{1/2}\right)dt
\end{eqnarray}
the last equality holds because, for $t>0$, $\mathbb{P}\left(\left\{T\right\}_{+}>t\right)=\mathbb{P}\left(T>t\right)$ since $\mathbb{P}\left(T>t \wedge T<0\right)=0$.\\

From the integrand $\mathbb{P}\left(|\tilde{r}_{h'}(x)-\mathbb{E}(\tilde{r}_{h'}(x))|>V_1 (h')+t^{1/2}\right)$ in equation \eqref{AARiesgoControl}, it is observed that,

\begin{equation*}
\tilde{r}_{h'}(x)-\mathbb{E}(\tilde{r}_{h'}(x))= 
\sum_{i=1}^n \left( G(X_i ,\varepsilon_i )-\mathbb{E}[G(X_i ,\varepsilon_i)] \right)=
\sum_{i=1}^n Z_i=
S_n
\end{equation*}
where the $Z_i$ and $S_n$ are defined in Proposition \ref{AATDBernstein}. Then, inequality \eqref{AADBernstein} is satisfied, where $\mathcal{A}_n=\frac{A_{1}}{nh'} + B_{8} \frac{(\log n)^{-\frac{1}{2}}}{nh'}$ is an upper bound for $Var(S_n)$ and $\mathcal{B}_n =B\frac{M_n}{nh'}$.\\

At this point, the analysis of $\Delta_1(h')$ is halted, to demonstrate that by taking $s(u)=\sqrt{2\mathcal{A}_n u}+4\mathcal{B}_n (2u)^3$ for each $u\in\mathbb{R}^+$, with $\mathcal{A}_n=\frac{A_{1}}{nh'} + B_{8} \frac{(\log n)^{-\frac{1}{2}}}{nh'}$ and $\mathcal{B}_n =B\frac{M_n}{nh'}$, it is satisfied that $s(\gamma|\log h'|) \leq V_1 (h')$ for $n$ sufficiently large.

\begin{eqnarray} \label{AACotaTV2_0}
\nonumber
s(\gamma|\log h'|) &=& \sqrt{2\left( \frac{A_{1}}{nh'} + B_{8} \frac{(\log n)^{-\frac{1}{2}}}{nh'} \right) \gamma|\log h'|}+4B\frac{M_n}{nh'} (2\gamma|\log h'|)^3\\
\nonumber
&\leq&
\sqrt{2\left( \frac{A_{1}}{nh'} + B_{8} \frac{(\log n)^{-\frac{1}{2}}}{nh'} \right) \gamma\log n} + 4(2\gamma)^3 B(\log n)^3\frac{M_n}{nh'}\\
\nonumber
&\leq&
\sqrt{ 2\gamma A_{1}\frac{\log n}{nh'} } + \sqrt{ 2\gamma B_{8} \frac{(\log n)^\frac{1}{2}}{nh'} } + 4(2\gamma)^3 B(\log n)^3 (\sigma +r_{\sup})\frac{\log n}{nh'}\\
\nonumber
&=&
\sqrt{ 2\gamma A_{1}\frac{\log n}{nh'} } \left( 1 + \sqrt{ \frac{B_{8}}{A_{1}} (\log n)^{-\frac{1}{2}} } \right) + 4(2\gamma)^3 B (\sigma +r_{\sup})\frac{(\log n)^4}{nh'}\\
\end{eqnarray}

Since $\frac{(\log n)^8}{n} \leq h'$, then $\frac{(\log n)^8}{nh'} \leq 1$, implying $\frac{(\log n)^4}{\sqrt{nh'}} \leq 1$. From the latter, the following inequality is established,

$$4(2\gamma)^3 B (\sigma +r_{\sup})\frac{(\log n)^4}{nh'} = \frac{(\log n)^4}{\sqrt{nh'}} \frac{4(2\gamma)^3 B (\sigma +r_{\sup})}{\sqrt{nh'}} \leq \frac{4(2\gamma)^3 B (\sigma +r_{\sup})}{\sqrt{nh'}}$$

Given this last inequality and the inequality \eqref{AACotaTV2_0}, we have,

\begin{align} \label{AACotaTV2_1}
\nonumber
s(\gamma|\log h'|) &\leq \sqrt{ 2\gamma A_{1}\frac{\log n}{nh'} } \left( 1 + \sqrt{ \frac{B_{8}}{A_{1}} (\log n)^{-\frac{1}{2}} } \right) + \frac{4(2\gamma)^3 B (\sigma +r_{\sup})}{\sqrt{nh'}}\\
&=
\sqrt{ 2\gamma A_{1}\frac{\log n}{nh'} } \left( 1 + \sqrt{ \frac{B_{8}}{A_{1}} }(\log n)^{-\frac{1}{4}} + \frac{4(2\gamma)^3 B (\sigma +r_{\sup})}{\sqrt{2\gamma A_{1}}}(\log n)^{-\frac{1}{2}} \right)
\end{align}

Since $\delta_n =(\log n)^{-\frac{1}{5}}$, it is clear that $(\log n)^{-\frac{1}{2}} \leq (\log n)^{-\frac{1}{4}} \leq \delta_n$ for $n\geq 3$. Therefore, $\frac{(\log n)^{-\frac{1}{2}}}{\delta_n} \searrow 0$ and $\frac{(\log n)^{-\frac{1}{4}}}{\delta_n} \searrow 0$. On the other hand, for any pair of constants $\alpha$ and $\beta$ in $\mathbb{R}^+$, it holds that,

$$ 0 \leq \alpha \frac{(\log n)^{-\frac{1}{2}}}{\delta_n} \leq \frac{1}{2} \quad \text{y} \quad 0 \leq \beta \frac{(\log n)^{-\frac{1}{4}}}{\delta_n} \leq \frac{1}{2}$$
for all $n$ sufficiently large. Therefore,

$$ 0 \leq \alpha (\log n)^{-\frac{1}{2}} \leq \frac{\delta_n}{2} \quad \text{y} \quad 0 \leq \beta (\log n)^{-\frac{1}{4}} \leq \frac{\delta_n}{2}.$$

Now, based on this last result and the inequality \eqref{AACotaTV2_1}, we have, for $n$ sufficiently large,

\begin{eqnarray} \label{AACotaTV2_2}
\nonumber
s(\gamma|\log h'|) &\leq& \sqrt{ 2\gamma A_{1}\frac{\log n}{nh'} } \left( 1 + \frac{\delta_n}{2} + \frac{\delta_n}{2} \right)\\
\nonumber
&=&
\sqrt{ 2\gamma A_{1}\frac{\log n}{nh'} } ( 1 + \delta_n )\\
&=&
V_1 (h').
\end{eqnarray}

We consider the functions $\lambda(u)=\sqrt{2\mathcal{A}_n u}+\mathcal{B}_n (2u)^3$, from Corollary \ref{AACDBernstein}, and $s(u)=\sqrt{2\mathcal{A}_n u}+4\mathcal{B}_n (2u)^3$ for each $u\in\mathbb{R}^+$. By making the change of variable $t^\frac{1}{2} =s(u)$ in equation \eqref{AARiesgoControl}, where $t=s^2 (u)$ and $dt=2s(u)s'(u)du$, we obtain,

\begin{eqnarray*} 
    \Delta_1(h') &=& \int_0^{+\infty}\mathbb{P}\left(|S_n|>V_1 (h') +s(u)\right)2s(u)s'(u)du.
\end{eqnarray*}
Based on the above equation and the inequality \eqref{AACotaTV2_2}, we have,

\begin{eqnarray} \label{AARiesgoControl1}
    \Delta_1(h') &\leq& \int_0^{+\infty}\mathbb{P}\left(|S_n|>s(\gamma|\log h'|)+s(u)\right)2s(u)s'(u)du.
\end{eqnarray}
Since the square root function is subadditive and $(p+q)^3 \leq 4(p^3 +q^3)$, we have,

\begin{eqnarray*}
    s(p)+s(q) &=& \sqrt{2\mathcal{A}_n p}+4\mathcal{B}_n (2p)^3 + \sqrt{2\mathcal{A}_n q}+4\mathcal{B}_n (2q)^3\\
    &\geq&
    \sqrt{2\mathcal{A}_n (p+q)}+\mathcal{B}_n (2(p+q))^3\\
    &=&
    \lambda(p+q).
\end{eqnarray*}
Based on inequality \eqref{AARiesgoControl1}, this last result, and Corollary \ref{AACDBernstein}, we have,

\begin{eqnarray} \label{AARiesgoControl2}
    \nonumber
    \Delta_1(h') &\leq& \int_0^{+\infty}\mathbb{P}\left(|S_n|>\lambda(\gamma|\log h'| +u)\right)2s(u)s'(u)du\\
    \nonumber
    &\leq&
    \int_0^{+\infty} \exp \left(-\frac{\gamma|\log h'|+u}{2}\right) 2s(u)s'(u)du\\
    &=&
    \exp \left(-\frac{\gamma|\log h'|}{2}\right)\int_0^{+\infty} \exp \left(-\frac{u}{2}\right) 2s(u)s'(u)du
\end{eqnarray}

\begin{itemize}
    \item Since $s(u) =\sqrt{2\mathcal{A}_n}\sqrt{u}+4 {B}_n 2^3 u^3$ and $s'(u)= \sqrt{2\mathcal{A}_n}\frac{1}{2\sqrt{u}}+4\mathcal{B}_n 2^3 3u^2$, then $s'(u)u \leq 3s(u)$, implying $2s'(u) s(u) \leq 6\frac{s^2 (u)}{u}$.
    
    \item Since $s^2 (u)=\left[ \sqrt{2\mathcal{A}_n u} +4\mathcal{B}_n (2u)^3\right]^2 \leq 2\left[ 2\mathcal{A}_n u +4^2 \mathcal{B}^2_n (2u)^6 \right]$, then $\frac{s^2 (u)}{u} \leq 2^2\mathcal{A}_n + 2^{11} \mathcal{B}^2_n u^5$.
\end{itemize}
Thus, based on these last two results and inequality \eqref{AARiesgoControl2}, we have, 

\begin{eqnarray}
    \nonumber
    \Delta_1(h') &\leq& \exp \left(-\frac{\gamma|\log h'|}{2}\right)\int_0^{+\infty} 6(2^2\mathcal{A}_n + 2^{11} \mathcal{B}^2_n u^5) \exp \left(-\frac{u}{2}\right) du\\
    \nonumber
    &=&
    3e^{\left\{\frac{\gamma }{2}\log h'\right\}}\left\{\int_0^{\infty}2^3\mathcal{A}_n e^{-\frac{u}{2}}du + \int_0^{\infty} 2^{12}\mathcal{B}_n^2 u^5 e^{-\frac{u}{2}}du \right\}\\
    \nonumber
    &=&
    3e^{\left\{\log ({h'}^\frac{\gamma }{2})\right\}}\left\{2^3\mathcal{A}_n \lim_{b \rightarrow \infty} \left[-2e^{-\frac{u}{2}}\right]_0^b + 2^{13}\mathcal{B}_n^2\int_0^{\infty}  (2w)^5 e^{-w}dw \right\}\\
     \nonumber
    &=&
    3(h')^\frac{\gamma }{2}\left\{2^4 \mathcal{A}_n + 2^{18}\mathcal{B}_n^2\int_0^{\infty}  w^5 e^{-w}dw \right\}\\
    \nonumber
    &=&
    3\left\{2^4 \mathcal{A}_n + 5!2^{18}\mathcal{B}_n^2 \right\}(h')^\frac{\gamma }{2}.
\end{eqnarray}

By summing over each $h'\in\mathcal{H}$ in the previous inequality, substituting the expressions for $\mathcal{A}_n$ and $\mathcal{B}_n$ provided in Proposition \ref{AATDBernstein}, and using the fact that $M_n \leq (\sigma +r_{\sup})\log n$ and $\frac{1}{nh'} \leq \frac{1}{\log n}$, we obtain,
\begin{eqnarray*}
    \lefteqn{\sum_{h'\in \mathcal{H}} \Delta_1(h')}\\
    \nonumber
    &\leq& \sum_{h'\in \mathcal{H}} 3\left\{2^4 \left[ \frac{A_{1}}{nh'} + B_{8} \frac{(\log n)^{-\frac{1}{2}}}{nh'} \right] + 5!2^{18}\left[B\frac{M_n}{nh'} \right]^2 \right\}(h')^\frac{\gamma }{2}\\
    &\leq&
    \sum_{h'\in \mathcal{H}} \left\{2^4 3\left[ \frac{A_{1}}{n} + B_{8} \frac{(\log n)^{-\frac{1}{2}}}{n} \right] (h')^{\frac{\gamma }{2}-1} + 5!2^{18} 3 \frac{B^2 (\sigma + r_{\sup})^2}{nh'}\frac{(\log n)^2}{nh'} (h')^{\frac{\gamma }{2}} \right\}\\
    &\leq&
    \sum_{h'\in \mathcal{H}} \left\{2^4 3\left[ \frac{A_{1}}{n} + B_{8} \frac{(\log n)^{-\frac{1}{2}}}{n} \right] (h')^{\frac{\gamma }{2}-1} + 5!2^{18} 3 B^2 (\sigma + r_{\sup})^2 \frac{\log n}{n} (h')^{\frac{\gamma }{2}-1} \right\}\\
     &=&
     \left\{2^4 3\left[ \frac{A_{1}}{n} + B_{8} \frac{(\log n)^{-\frac{1}{2}}}{n} \right] + 5!2^{18} 3 B^2 (\sigma + r_{\sup})^2 \frac{\log n}{n}  \right\}\sum_{h'\in \mathcal{H}} (h')^{\frac{\gamma }{2}-1}.
\end{eqnarray*}
Now, since $\mathcal{H}=\{e^{-i}\}_{i=0}^{M}\cap [h_{\min},h_{\max}]$ and $\gamma-2 >0$ since $\gamma>2$, the summation on the right side of the previous inequality is rewritten as follows.
\begin{eqnarray} \label{AACotaSumAh'}
    \nonumber
    \lefteqn{\sum_{h'\in \mathcal{H}} \Delta_1(h') }\\
    \nonumber
    &\leq&
     \left\{2^4 3\left[ \frac{A_{1}}{n} + B_{8} \frac{(\log n)^{-\frac{1}{2}}}{n} \right] + 5!2^{18} 3 B^2 (\sigma + r_{\sup})^2 \frac{\log n}{n}  \right\}\sum_{h'\in \mathcal{H}} (h')^{\frac{\gamma-2}{2}}\\
    \nonumber
    &\leq&
    \left\{2^4 3\left[ \frac{A_{1}}{n} + B_{8} \frac{(\log n)^{-\frac{1}{2}}}{n} \right] + 5!2^{18} 3 B^2 (\sigma + r_{\sup})^2 \frac{\log n}{n}  \right\} \sum_{i=0}^{\infty}  \left(e^{-\frac{\gamma-2}{2}}\right)^{i}\\
    \nonumber
    &\leq&
    \left\{\frac{2^4 3}{1-e^{-\frac{\gamma-2}{2}}}\left[ \frac{A_{1}}{n} + B_{8} \frac{(\log n)^{-\frac{1}{2}}}{n} \right] + \frac{5!2^{18} 3 B^2}{1-e^{-\frac{\gamma-2}{2}}} (\sigma + r_{\sup})^2 \frac{\log n}{n} \right\}\\
    &\leq&
   B_{10}\frac{\log n}{n}
\end{eqnarray}
where $B_{10} = \frac{2^4 3}{1-e^{-\frac{\gamma-2}{2}}}(A_{1} + B_{8}) + \frac{5!2^{18} 3 B^2}{1-e^{-\frac{\gamma-2}{2}}} (\sigma + r_{\sup})^2$. 

Based on inequalities \eqref{AATemporal1}, \eqref{AACotaSumBh'}, and \eqref{AACotaSumAh'}, we have,

\begin{equation} \label{AACotaET_1^2}
\mathbb{E}(T_1^2) \leq 2\left(B_{10}\frac{\log n}{n} + \frac{B_{9}}{2n} \right)
=
A_5\frac{\log n}{n}.
\end{equation}

\subsection{Proof of Proposition \ref{AAPropET_1^2yV_1}, (ii).}\label{AAdemoT2}

To study $\mathbb{E}(T_2^2)$, a similar approach to $T_1$ is used, where

\begin{equation}
\label{AAET_2^2Cota0}
\mathbb{E}(T_2^2) \leq\sum_{h'\in\mathcal{H}}\mathbb{E}\left[\left\{|\hat{r}_{hh'}(x)
-\mathbb{E}(\hat{r}_{hh'}(x))|-V_{2}(h')\right\}_+^2\right]
\end{equation}

Similar to the proof of Proposition \ref{AAPropET_1^2yV_1} (i), the truncation auxiliary estimator is defined as follows.
$$
\Tilde{r}_{hh'}(x)=\frac{1}{n}\sum_{k=1}^nY_kK_{hh'}(x-X_k)g^{-1}(X_k)\mathbf{1}_{\{|Y_k |\leq M_n\}}
$$
where $K_{hh'}(\cdot)=K_{h}\ast K_{h'}(\cdot)$, for $h, h'\in\mathcal{H}$.\\

Similarly, one has
\begin{eqnarray*}
\mathbb{E}\left( \{|\hat{r}_{hh'}(x)
-\mathbb{E}(\hat{r}_{hh'}(x))|-V_2 (h')\}_+^2\right) \leq 2\Tilde{\Delta}_1(h') + 2\Tilde{\Delta}_2(h'),
\end{eqnarray*}
where $$\Tilde{\Delta}_1(h')=\mathbb{E}\left( \{|\Tilde{r}_{hh'}(x)
-\mathbb{E}(\Tilde{r}_{hh'}(x))|-V_2 (h')\}_+^2\right)$$
and
$$\Tilde{\Delta}_2(h')=\mathbb{E}\left( \left\{\hat{r}_{hh'}(x)-\Tilde{r}_{hh'} (x)-\mathbb{E}\left(\hat{r}_{hh'}(x)-\Tilde{r}_{hh'}(x)\right)\right\}^2\right).$$

We have
\begin{eqnarray}
\label{AAET_2^2Cota1}
\mathbb{E}(T_2^2) &\leq& 2\left( \sum_{h'\in\mathcal{H}}\Tilde{\Delta}_1(h') + \sum_{h'\in\mathcal{H}}\Tilde{\Delta}_2(h') \right).
\end{eqnarray}

We proceed to bound the summation $\sum_{h'\in\mathcal{H}}\Tilde{\Delta}_2(h')$, rewriting
\begin{equation*}
\hat{r}_{hh'}(x)-\Tilde{r}_{hh'}(x) =
\frac{1}{n}\sum_{k=1}^n\tilde{\eta}_k
\end{equation*}
where $\tilde{\eta}_k = Y_kK_{hh'}(x-X_k)g^{-1}(X_k)\mathbf{1}_{\{|Y_k |> M_n\}}$.\\

It is observed that the only difference between the term $\Tilde{\eta}_k$ and the term $\eta_k$ obtained in Proposition \ref{AAPropET_1^2yV_1} is the kernel used. In this case, it is $K_{hh'}$, while in the previous case, it was $K_{h'}$. This difference does not play a determining role in the following proofs. It is only necessary to take into consideration that whenever $\|K_{h'}\|_\infty$ appears, one should actually have $\|K_{hh'}\|_\infty$. Additionally, the following should be noted,
\begin{eqnarray*}
|K_{h'h}(x-X)| &=& |K_{h'}\ast K_{h}(x-X)|\\
&\leq&
\int_\mathbb{R}|K_{h'}(x-X-t) K_{h}(t)|dt\\
&=&
\|K_{h'}\|_\infty\int_\mathbb{R}|K_{h}(t)|dt\\
&=&
\|K_{h'}\|_\infty \|K\|_1
\end{eqnarray*}
hence, in the results obtained in Proposition \ref{AAPropET_1^2yV_1} (i), $\|K\|_\infty$ should be replaced by $\|K\|_\infty \|K\|_1$. With the previously stated and inequality \eqref{AACotaSumBh'}, we have,
\begin{equation} \label{AACotaSumTildBh'}
    \sum_{h'\in\mathcal{H}}\Tilde{\Delta}_2(h')  \leq  \frac{B_9\|K\|_1^2}{2n}.
\end{equation}

Now we proceed to bound the summation $\sum_{h'\in\mathcal{H}}\Tilde{\Delta}_1(h')$. We have a similar argument to the proof of part (i):
\begin{eqnarray*}
\label{AATild_Ah'Cota0}
\lefteqn{\mathbb{E}\left(\left\{|\tilde{r}_{hh'}(x)-\mathbb{E}(\tilde{r}_{hh'}(x))|-V_2 (h')\right\}_+^2\right)}\\
\nonumber
&=&
\int_0^{+\infty}\mathbb{P}\left(|\tilde{r}_{hh'}(x)-\mathbb{E}(\tilde{r}_{hh'}(x))|>V_2 (h')+t^{1/2}\right)dt\\
&=&
\int_0^{+\infty}\mathbb{P}\left(|\tilde{S}_n|>V_2 (h')+t^{1/2}\right)dt,
\end{eqnarray*}
where $\Tilde{S}_n=\sum_{i=1}^n \Tilde{Z}_i$, $\Tilde{Z}_i =\Tilde{G}(X_i ,\varepsilon_i )-\mathbb{E}[\Tilde{G}(X_i ,\varepsilon_i)]$, $\Tilde{G}(X_i ,\varepsilon_i) = \frac{1}{n}Y_i K_{hh'}(x-X_i)g^{-1}(X_i)\mathbf{1}_{\{|Y_i |\leq M_n\}}$.\\

Again, it is observed that in the following proofs, the fact that the definition of $\Tilde{G}(\cdot , \cdot )$ is in terms of the kernel $K_{hh'}$ instead of the kernel $K_{h'}$ is not crucial. It is only necessary to consider that whenever $\|K_{h'}\|_2$, $\|K_{h'}\|_1$, and $\|K_{h'}\|_{\infty}$ appear, one should actually have $\|K_{hh'}\|_2$, $\|K_{hh'}\|_1$, and $\|K_{hh'}\|_{\infty}$ respectively, and to use the fact that
\begin{itemize}
\item $\|K_{hh'}\|_\infty \leq \|K_{h'}\|_\infty \|K\|_1$
\item $\|K_{hh'}\|_2 \leq \|K_{h'}\|_2 \|K_h\|_1=\|K_{h'}\|_2 \|K\|_1$
    \item $\|K_{hh'}\|_1 \leq \|K_{h'}\|_1 \|K_h\|_1=\|K\|_1^2$.
\end{itemize}

So, each time $\|K\|_2$, $\|K\|_1$, and $\|K\|_{\infty}$ appear in the results of Proposition \ref{AAPropET_1^2yV_1} (i), and Proposition \ref{AATDBernstein}, they should be replaced by $\|K\|_2\|K\|_1$, $\|K\|_1\|K\|_1$, and $\|K\|_{\infty}\|K\|_1$ respectively. Considering the aforementioned and inequality \eqref{AACotaSumAh'}, we have,
\begin{eqnarray} \label{AACotaSumTildAh'}
    \sum_{h'\in \mathcal{H}} \Tilde{\Delta}_1(h') &\leq& B_{10}\|K\|_1^2\frac{\log n}{n}.
\end{eqnarray}
Now, based on the inequalities \eqref{AACotaSumTildBh'} and \eqref{AACotaSumTildAh'}, we have,
\begin{eqnarray} \label{AACotaET_2^2}
\nonumber
\mathbb{E}(T_2^2) &\leq& 
A_5\|K\|_1^2\frac{\log n}{n}.
\end{eqnarray}

\section*{Acknowledgments}
 K.\ Bertin is supported by FONDECYT regular grants 1221373 and 1230807 from ANID-Chile, and the Centro de Modelamiento Matemático (CMM) BASAL fund FB210005 for centers of excellence from ANID-Chile. L.\ Fermín is supported by FONDECYT regular grant 1230807 from ANID-Chile. M.\ Padrino is supported by FONDECYT regular grant 1221373 from ANID-Chile and PhD grant CONICYT –
PFHA / Doctorado Nacional 2019 – 21191358.

This final version of this article will be published in Statistics.



\end{document}